\documentclass[11pt]{amsart}
\usepackage{amsmath}
\usepackage{amssymb}
\usepackage{dsfont}
\usepackage{url}
\usepackage{graphicx}
\newtheorem{thmm}{Theorem}
\newtheorem{thm}{Theorem}[section]
\newtheorem{lem}[thm]{Lemma}
\newtheorem{prop}[thm]{Proposition}
\newtheorem{cor}[thm]{Corollary}
\newtheorem{defn}[thm]{Definition}
\newtheorem{conj}[thm]{Conjecture}
\newtheorem{remark}[thm]{Remark}

\newcommand{\bZ}{{\mathbb Z}} %cph changed from \mathbf
\newcommand{\bQ}{{\mathbb Q}}
\newcommand{\bR}{{\mathbb R}}
\newcommand{\bC}{{\mathbb C}}
\newcommand{\bN}{{\mathbb N}}
\newcommand{\bF}{{\mathbb F}}
\newcommand{\bT}{{\mathbb T}}

\newcommand{\dsone}{\mathds1}

\renewcommand{\mod}{\;\operatorname{mod}}
\newcommand{\ord}{\operatorname{ord}}

\newcommand{\Cinf}{C^\infty}

\newcommand{\tr}{\operatorname{tr}}
\newcommand{\OPN}{\operatorname{Op}_N}
\newcommand{\Var}{\operatorname{Var}}

\newcommand{\disc}{\operatorname{disc}}

\newcommand{\eps}{\varepsilon}
\newcommand{\HN}{\mathcal H_N}
\newcommand{\UN}{U_N}
\newcommand{\intinf}{\int_{-\infty}^\infty}
\newcommand{\CA}{{\mathcal C}_A}

\begin{document}

\title[Fluctuations of matrix elements]
{On the fluctuations of matrix elements of the quantum cat map}
\author{Lior Rosenzweig}%
\address{Raymond and Beverly Sackler School of Mathematical Sciences,
Tel Aviv University, Tel Aviv 69978, Israel}

\email{rosenzwe@post.tau.ac.il}%

%\thanks{}%
%\subjclass{}%
%\keywords{}%

%\date{}%
%\dedicatory{}%
%\commby{}%
% ----------------------------------------------------------------
\begin{abstract}
We study the fluctuations of the diagonal matrix elements of the
quantum cat map about their limit. We show that after suitable normalization,  the fifth centered moment for the Hecke basis vanishes in the semiclassical limit, confirming in part a conjecture of Kurlberg and Rudnick.

We also study  sums of matrix elements lying in short windows. For
observables with zero mean, the first moment of these sums is zero ,
and the variance was determined by the author with Kurlberg and
Rudnick. We show that if the window is sufficiently small in terms
of Planck's constant, the third moment vanishes if we normalize so
that the variance is of order one.
\end{abstract}
\maketitle
% ----------------------------------------------------------------
\section{Introduction}
The study of quantum wave functions of classically chaotic systems
has been extensively studied in recent years. One well known result
is that in the mean square sense the matrix elements of smooth
observables concentrate around the classical average of the
observable in the semiclassical limit
\cite{Shn},\cite{Zel},\cite{Dever}. This is known as the "Quantum
Ergodicity Theorem". The problem of whether all matrix elements
converge to the classical average (the "Quantum Unique Ergodicity"
problem) has no general result so far. This has been extensively
studied, and in some arithmetic cases both positive (when
considering desymmetrized eigenfunctions) answers (cf
\cite{Lin,SH,Sound,KR_Hecke,Kelmer}) and negative answers (cf
\cite{FNDe,Kelmer,Kel2}) have been given.

Another important property is the distribution of the matrix
elements. It was suggested by Feingold and Peres \cite{FP86} that
for generic systems with $D$ degrees of freedom, the variance of the
matrix elements about their mean decays with Planck's constant
$\hbar$ as  $\hbar^D$, with a prefactor given in terms of the
autocorrelation function of the classical observable. Furthermore in
\cite{Eck} Eckhart {\it{et al}} predict that after normalizing the
fluctuations of the matrix elements, they have a limiting Gaussian
distribution about their limit with the same expected value and
variance. Some arithmetical models were found to deviate from these
predictions \cite{LS,KR_Annals,Kelmer}.

In this paper we study properties of these fluctuations for the
quantum cat map. To describe these properties we first recall the
model.
\subsection{The Quantum cat map}
The quantized cat map is a model quantum system with chaotic
classical analogue, first investigated by Hannay and Berry \cite{HB}
and studied extensively since, see e.g. \cite{Keating91, DEGI,
KR_Hecke, FNDe, Rudnick-montreal}. While the classical system
displays generic chaotic properties, the quantum system behaves
non-generically in several aspects, such as the statistics of the
eigenphases, and the value distribution of the eigenfunctions
\cite{KR-imrn}.

We review some of the details of the system in a form suitable for
our purposes, see e.g. \cite{DEGI, KR_Hecke, Rudnick-montreal}. Let
$A$ be a linear hyperbolic toral automorphism, that is, $A\in
SL_2(\bZ)$ is an integer unimodular matrix with distinct real
eigenvalues. We assume $A\equiv I\mod 2$. Iterating the action of
$A$ on the torus $\bT^2=\bR^2/\bZ^2$ gives  a dynamical system,
which is highly chaotic. The quantum mechanical system includes an
integer $N\geq 1$, the inverse Planck constant, (which we will take
to be prime), an $N$-dimensional state space $\HN\simeq
L^2(\bZ/N\bZ)$, and a unitary map $U=\UN(A)$ of $\HN$, which is the
quantization $A$. The eigenvalues and the dimension of the
eigensapces of $U$ are related to the order of $A$ modulo $N$. Let
$\ord(A,N)$ be the least integer $r\geq 1$ for which $A^r\equiv I
\mod N$. When $N$ is prime the distinct eigenphases $\theta_j$ are
evenly spaced (with at most one exception) with spacing
$1/\ord(A,N)$, and in fact,
%after removing  a phase in the quantization of $U_N(A)$,
the distinct eigenphases are all of the form $j/\ord(A,N)$. The
eigenspaces all have the same dimension (again with at most one
 exception) which is $(N\pm 1)/\ord(A,N)$.

For fixed small $\epsilon>0$, as $N\to\infty$ through a sequence of
values such that  $\ord(A,N)>N^\epsilon$ all the matrix elements
converge to the phase space average $\int_{\bT} f(x)dx$ of the
observable $f$ \cite{KR00, Bou07} %It is important to keep in mind that this is not always be the case:
(However, note that there are ``scars'' found for values of $N$
where $\ord(A,N)$ is logarithmic in $N$, see \cite{FNDe}.) The
condition on $\ord(A,N)$ is valid for most values of $N$ (in fact
$\ord(A,N)>N^{1/2+o(1)}$ for almost all $N$, c.f. \cite[Lemma
15]{KR00}), Moreover, it was shown by Kurlberg in \cite{Kur03} that
assuming GRH, for almost all primes $N$ $\ord(A,N)\gg N/b(N)$ for
any function $b(x)$ tending to infinity more slowly than $\log x$,
and for almost all values of $N$, $\ord(A,N)\geq N^{1-\eps}$.

%Fixing a smooth real-valued  observable $f\in C^\infty(\bT)$,
%%which in fact we will take to be  a trigonometric polynomial (that is
%%having only finitely many Fourier modes),
%(which we will assume has zero mean: $ \int_{\bT} f(x)dx =0\;, $),
%and let $\OPN(f)$ be its quantization, which is a self-adjoint
%operator on $\HN$.
%(which we will in fact assume is a trigonometric polynomial?)
 In \cite{KR_Hecke} Kurlberg and Rudnick
introduced a group of unitary operators, the Hecke group, that
commutes with $U$. It is shown in \cite{KR_Hecke} that if
$\{\psi_N\}$ is a sequence of Hecke eigenfunction (a joint
eigenfunctions of $U$ and all elements of the Hecke group), then for
any smooth function $f\in\Cinf(\bT^2)$ the matrix elements
$\langle\OPN(f)\psi_N,\psi_N\rangle$ converge to the space average
$\int_{\bT^2}f$. In \cite{KR_Annals} they raise a conjecture about
the fluctuation of the matrix elements around the limit for a fixed
function. The operator $\OPN(f)$ is decomposed by the Fourier
decomposition of $f$, that is if
$f(x)=\sum_{n\in\bZ^2}\hat{f}(n)e(nx)$, then
$\OPN(f)=\sum_{n\in\bZ^2}\hat{f}(n)\OPN(e(nx))$. They conjecture
that for fixed $0\ne n\in\bZ^2$ the set
$\langle\OPN(e(nx))\psi_j,\psi_j\rangle$ becomes equidistributed
with respect to the Sato-Tate measure, and after considering
symmetries of the system these sets become independent for different
choices of $n$ (a more precise explanation is given in section
\ref{sec:Hecke_mtrx_lmnts}). Agreement with this conjecture is shown
in figures
\ref{fig:Catmap_distribution_single},\ref{fig:Catmap_distribution_double}.
In figure \ref{fig:Catmap_distribution_single} the cumulative
distribution function (cdf) of the fluctuations of the matrix
elements for the fixed function $f(x)=e(x+y)$ is shown compared with
the cdf of a random variable with Sato-Tate distribution (the
probability density function in this case is
$p(x)=\frac1{2\pi}\sqrt{4-x^2}$). In figure
\ref{fig:Catmap_distribution_double} the fixed function is
$f(x)=e(x+y)+e(x+2y)$. In this case the expected limiting
distribution is of the sum of two independent random variables with
Sato-Tate distribution, and again the cdf of the matrix elements
shows high agreement with the conjecture. The matrix used in both
cases is $\begin{pmatrix}7&-2\\4&-1\end{pmatrix}$
\begin{figure}[ht]
   \centerline{ \includegraphics[width=14cm,height=10cm]{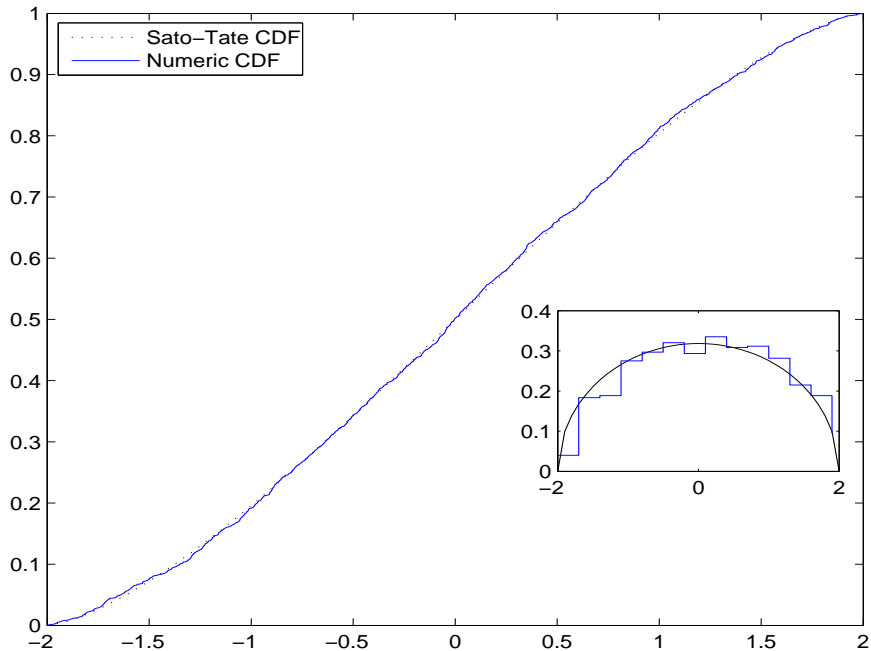}}
   \caption{Cumulative distribution function of Hecke eigenbasis, with $f(x)=e(x+y)$, $N=1997$ compared to Sato-Tate comulative distribution function}
   \label{fig:Catmap_distribution_single}
       \end{figure}
\begin{figure}[ht]
   \centerline{ \includegraphics[width=14cm,height=10cm]{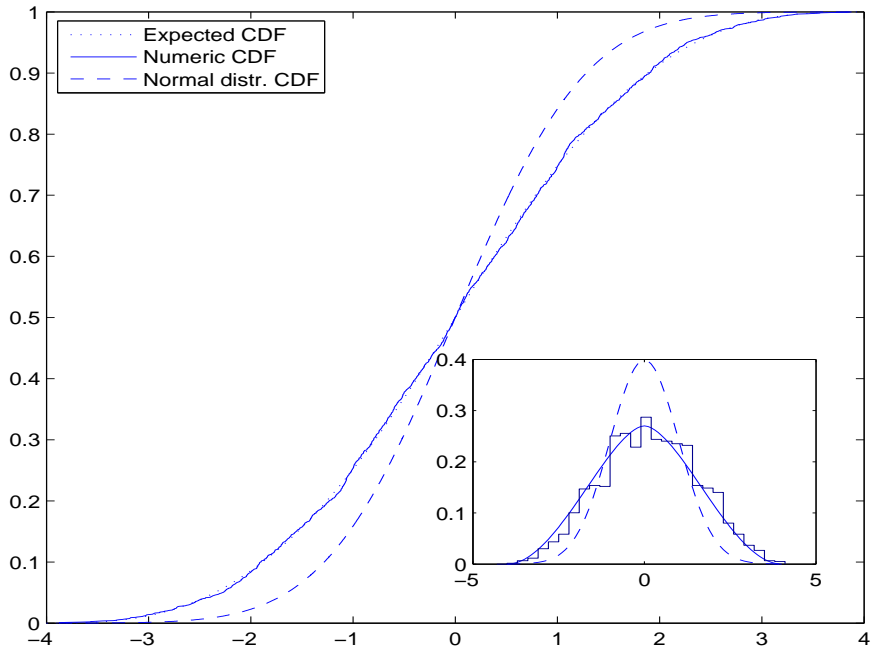}}
   \caption{Cumulative distribution function of Hecke eigenbasis compared to sum of two independent Sato-Tate's and to a standard normal distribution}
   \label{fig:Catmap_distribution_double}
       \end{figure}
Another way to study the fluctuations of the matrix elements, is by
%about $\int_{\TT} f(x)dx=0$,
studying the sum of diagonal matrix elements of $\OPN(f)$ over
eigenphases lying in a random window of length $1/L$ around
$\theta$. More generally we consider a window function, constructed
by taking a fixed non-negative and even function $h \in L^2([-\frac
12,\frac 12])$ and setting $h_L(\theta) :=\sum_{m\in \bZ}
h(L(\theta-m))$, which is periodic and localized in an interval of
length $1/L$. We further normalize so that $\intinf h(x)^2dx=1$,
and hence $\int_0^1 h_L(\theta)^2 d\theta %= \frac 1L \intinf h(x)^2 dx
= 1/L$. Then set
\begin{equation}\label{defn P}
P(\theta) := \sum_{j=1}^N h_L(\theta-\theta_j) \langle
\OPN(f)\psi_j,\psi_j \rangle  \;.
\end{equation}
Note that $P(\theta)$ is independent of choice of basis, and in
particular it is real valued. An important case to consider is the
case where $f(x)$ is a trigonometrical function, we therefore denote
for $n\in\bZ^2$
\begin{equation*}
P_n(\theta) := \sum_{j=1}^N h_L(\theta-\theta_j) \langle
\OPN(e(nx))\psi_j,\psi_j \rangle  \;.
\end{equation*}
 In \cite{KRR} it was shown that if $\ord(A,N)\gg
N^{1/2}$ then $\Var(\sqrt{L} P)\sim C(f) +o(1)$ where $C(f)$ is a
constant depending on $f$ and the matrix $A$. This variance is the
same variance as the limiting variance of the distribution of Hecke
matrix elements.
\subsection{Results} In the following we present two results in the
study of the fluctuations of matrix elements. In section \ref{sec:P}
we study the fluctuations in short windows. In \cite{KRR} we showed
that unless $n,m$ satisfy an arithmetic condition, the corresponding
fluctuation functions, $P_n(\theta),P_m(\theta)$, become
uncorrelated. In this paper we generalize this result for any choice
of triple $n_1,n_2,n_3\in\bZ^2$. That is we show that as
$N\to\infty$ through primes
$$
\int_0^1P_{n_1}(\theta)P_{n_2}(\theta)P_{n_3}(\theta)d\theta=O(\frac
N{L^{3/2}})
$$
and in particular we prove the following theorem
\begin{thmm}
Let $A\in SL_2(\bZ)$ be a hyperbolic matrix satisfying $A\equiv
I\pmod2$. Fix $f\in C^\infty(\bT^2)$ of zero mean. Assume
$L<2\ord(A,N)$, then as $N\to\infty$ through split primes satisfying
$\ord(A,N)/N^{2/3}\to\infty$, the third moment of $P(\theta)$
satisfies
\begin{equation}
\int_0^1\left(\sqrt LP(\theta)\right)^3d\theta=o(1)
\end{equation}
\end{thmm}
This results is consistent with a conjecture that $\sqrt{L}P$  has a
Gaussian distribution (see section \ref{sec:discussion}).

In section \ref{sec:Hecke_mtrx_lmnts} we show agreement with the
expected Sato-Tate limiting distribution and independent behaviour
of the fluctuation of the martix elements for a fixed function.
According to \cite{KR_Annals}, the normalized matrix coefficient
$$ \sqrt{N} \langle Op_N(f)\psi_j,\psi_j \rangle - \int_ f $$
should be distributed like a weighted sum of traces of independent
random matrices in $SU(2)$. In \cite{KR_Annals}, the second and
fourth moments are computed and shown to be consistent with this
conjecture. We show that the fifth moment vanishes, in accordance to
the conjecture:
\begin{thmm}\label{thm:mtrx_lmns_ffth_mmnt_gnrl}
Let $A\in SL_2(\bZ)$, $U_N(A)$ its quantization and
$\{\psi_j\}_{j=1}^N$ a Hecke Basis. Fix $f\in C^\infty(\bT^2)$. Then as $N\to\infty$ through primes,
$$
\frac1N\sum_{j=1}^N\left(\sqrt{N}(\langle\OPN(f))\psi_j,\psi_j\rangle-\int_{\bT^2}f)\right)^5=O_f(\frac1{\sqrt{N}})
$$
\end{thmm}
%We show that if $\{\psi_j\}$
%is a Hecke eigenbasis, and for any choice of quintuple
%$n_1,\dots,n_5\in\bZ^2$, then as $N\to\infty$ through primes
%$$
%\frac1N\sum_{j=1}^N\prod_{i=1}^5\langle\OPN(e(n_ix))%\psi_j,\psi_j\rangle=O(\frac1{\sqrt{N}})
%$$
%and in particular for any function $f$
%$$
%\frac1N\sum_{j=1}^N\left(\langle\OPN(f))%\psi_j,\psi_j\rangle-\int_{\bT^2}f\right)^5=O(\frac1{\sqrt{N}})
%$$
%These results also agrees with the expected Sato-Tate and
%independence behaviour.

The results presented here are corollaries from bounds of mixed
moments of a certain family of exponential sums. In sections
\ref{sec:Fam_exp_sum},\ref{sec:prf_exp_sum_3rd_mmnt},\ref{sec:prf_exp_sum_5th_mmnt}
we introduce this family and study mixed moments of its
distribution.

% We deal
%with two cases for the size of the window. The first case is when
%the matrix is of maximal order ($N\pm1$), and the windows is of size
%$L=N+o(1)$. In this case studying the fluctuations in such windows
%is in fact the same as studying the distribution of the matrix
%elements and the limiting distribution is the same. Another case is
%when $L=o(\ord(A,N))$ and $N^{1-\epsilon}\ll L$. In this case the
%number of matrix elements taken into the sum at every point is
%growing with as $N\to\infty$, and due to '{\it{expected}}' limiting
%independence behaviour of the matrix elements expect Gaussian
%limiting distribution. The expected independence is given by a
%conjecture on the behaviour of certain exponential sums over finite
%fields. Assuming this conjecture we are able to prove the following
%conjecture
%% general introduction and statement of main conjecture and
%%theorems.
%\begin{conj}\label{conj:gaussian_P}
%Let $A\in SL_2(\bZ)$ be a unimodular matrix with distinct
%eigenvalues such that $A\equiv I\pmod2$. Fix a smooth function $f\in
%C^\infty(\bT^2)$. Assume that for every $\epsilon>0$
%$N^{1-\epsilon}\ll L$, and $\frac{L}{\ord(A,N)}\to0$ as $N$ goes to
%infinity through primes, then $P(\theta)$ has Gaussian limiting
%distribution with mean 0 and $var(P)=\sum_{\nu\in\bZ}f^\sharp(\nu)$
%\end{conj}
\section*{Acknowledgements}
I warmly thank my Ph.D. advisor Ze\'ev Rudnick for his guidance and
support throughout this work. I thank Par Kurlberg and Dubi Kelmer
for long discussions and comments. I would also like to thank
Emmanuel Kowalski for helpful discussions about the theory of
exponential sums. This work was Supported by the Israel Science
Foundation grant No. 925/06. This work was carried out as part of
the author's Ph.D. thesis at Tel Aviv University, under the
supervision of Prof. Ze\'ev Rudnick.
\section{Background}\label{sec:background}
\subsection{Quantum mechanics on the torus}
We recall the basic facts of quantum mechanics on the torus which we
need in the paper,  see \cite{Rudnick-montreal,KR_Hecke} for further
details. Planck's constant is restricted to be an inverse integer
$1/N$, and the Hilbert space of states $\HN$ is $N$-dimensional,
which is identified with $L^2(\bZ/N\bZ)$ with the inner product
given by
\begin{equation*}
\langle \phi,\psi \rangle
 = \frac1N \sum_{Q\bmod N} \phi(Q) \, \overline\psi(Q) \;.
\end{equation*}

Classical observables, that is real-valued functions $f\in
C^\infty(\bT)$, give rise to quantum observables, that is
self-adjoint operators $\OPN(f)$ on $\HN$. To define these, one
starts with translation operators: For $n=(n_1,n_2)\in\bZ^2$ let
$T_N(n)$ be the unitary operator on $\HN$ whose action on a
wave-function $\psi\in \HN$ is
\begin{equation*} %\label{action of T(n)}
T_N(n)\psi(Q) = e^{\frac {i\pi n_1 n_2}N} e(\frac{n_2Q}N)\psi(Q+n_1)
\;.
\end{equation*}
For any smooth function $f\in C^\infty(\bT)$,  define $\OPN(f)$ by
$$
\OPN(f) = \sum_{n\in\bZ^2} \widehat f(n) T_N(n)
$$
where $\widehat f(n)$ are the Fourier coefficients of $f$. Below is
a list of properties of $\OPN(f)$:
\begin{enumerate}
\item
For $n=(n_1,n_2)\in\bZ^2$, denote $\epsilon(n)=(-1)^{n_1n_2}$, then
\begin{equation}\label{eq:trace_of_TN}
\tr(T_N(n))=\begin{cases}\epsilon(n)N & n=0\pmod N\\0 &
{\rm{otherwise}}\end{cases}
\end{equation}
and so for $f\in C^\infty(\bT^2)$
$$
\tr(\OPN(f))=N\int_{\bT^2}fdx+O(\frac1{N^\infty})
$$
\item
For $n,m\in\bZ^2$
\begin{equation}\label{eq:CCRN}
T_N(n)T_N(m)=e(\frac{\omega(m,n)}{2N})T_N(m+n)
\end{equation}
\end{enumerate}
%{\bf Properties of $\OPN(f)$}
%\begin{enumerate}
%\item
%For $(0,0)\ne n\in\bZ^2$ $\tr T_N(n)=0$ for $N$ large enough
%\item
%
%\end{enumerate}

\subsection{Quantized cat map: Definition and results}
For $B\in SL_2(\bZ)$ the quantized cat $U_N(B)$ is a unitary
operator on $\HN$ satisfying "Exact Egorov" property
$$U_N(B)^*\OPN(f)U_N(B)=\OPN(f\circ B)$$

In \cite{KR_Hecke} Kurlberg and Rudnick introduced a family of
commuting operators $\CA(N)$, called the Hecke group, which satisfy
that after taking joint eigenfunctions of all elements in $\CA(N)$,
then all corresponding matrix elements satisfy
$$|\langle\OPN(f)\psi,\psi\rangle-\int_{\bT^2}f|\ll N^{-1/4-\epsilon}$$
 and when $N$ is restricted to primes, Gurevich and Hadani showed in \cite{GH} that the rate of
convergence is in fact bounded by
$$|\langle\OPN(f)\psi,\psi\rangle-\int_{\bT^2}f|\leq C(f)N^{-1/2}.$$
We restrict our discussion from now on to $N$ prime. In this case,
all but a finite subset of the primes $A$ is diagonalizable over
either $\bF_N$ (the split case), or over $\bF_{N^2}$ (the inert
case). In the split case the group $\CA(N)$ is isomorphic to
$\bF_N^*$, and in the inert case it is isomorphic to $\bF_{N^2}^1$
the group of norm one elements in $\bF_{N^2}$. In \cite{KR_Annals}
Kurlberg and Rudnick exhibit some relations between Hecke matrix
elements (matrix elements corresponding to Hecke eigenfunctions).
For $A\in SL_2(\bZ)$ they introduced a quadratic form $Q(n)$ related
to $A$
$$Q(n)=\omega(n,nA)=cn_1^2+(d-a)n_1n_2-bn_2^2\;\;\;\;A=\begin{pmatrix}a&b\\c&d\end{pmatrix}
$$ where $\omega(x,y)=x_1y_2-x_2y_1$ is the
standard symplectic form. As this quadratic form plays a crucial
role in this paper, we list here some of its properties that were
proven in \cite{KRR}, \cite{KR_Annals}:
\begin{enumerate}
\item
Since $A$ is symplectic then $A$ preserves $Q(n)$, that is for all
$n\in\bZ$ $Q(nA)=Q(n)$. Moreover, the Hecke group $\CA(N)$ is
isomorphic to $SO(Q,\bZ/N\bZ)$. (\S2.3 in \cite{KR_Annals})
\item
Let $N$ be an odd prime, and let $A\in SL_2(\bZ)$ so that $(\tr
A)^2-4\neq 0\mod N$. Then the space of binary quadratic forms
preserved by $A$ (which contains $Q$) is one dimensional. (Lemma 2.1
in \cite{KRR})
\item
For $g\in \CA(N)$, $g\neq 1$, denote by
$$
q(x;g):= \omega(x(g-1)^{-1},x(g-1)^{-1}g).
$$
(Note that $q(\bullet,g)=0$ if $g=-I\mod N$). Then if $g\neq \pm I
\mod N$ then $q(\bullet;g)$ is a nonzero multiple of $Q$. Moreover,
\begin{equation}\label{formula for q}
q(x;g) = \frac{Q(x)}{\lambda_A-\lambda_A^{-1}}
\frac{1+\lambda}{1-\lambda}.
\end{equation}

 where $\lambda_A$ is a generator of $\CA(N)$ (\S2.4 in
\cite{KRR})
\end{enumerate}

 Using this quadratic form, for a smooth function $f(x)\in
C^{\infty}(\bT^2)$ they introduce
$$f^\sharp(\nu)=\sum_{n:Q(n)=\nu}(-1)^{n_1n_2}\hat f(n)$$
where $\hat f(n)$ are the Fourier coefficients of $f(x)$. They also
conjecture the following
\begin{conj}\label{conj:KR_equidistribution conjecture}
Let $A\in SL_2(\bZ)$, $U_N(A)$ its quantization and
$\{\psi_j\}_{j=1}^N$ a Hecke Basis. Let
$$
F_j^{(N)}=\sqrt
N\left(\langle\OPN(f)\psi_j,\psi_j\rangle-\int_{\bT^2}f\right)
$$
Then as $N\to\infty$ through primes, the limiting distribution of
the normalized matrix elements $F_j^{(N)}$ is that of the random
variable
$$
X_f:=\sum_{\nu\ne0}f^\sharp(\nu)\tr(U_\nu)
$$
where $U_\nu$ are independently chosen random matrices in $SU(2)$
endowed with Haar probability measure.
\end{conj}
(See further discussion and properties in section
\ref{sec:Hecke_mtrx_lmnts}.)
\subsection{Specific definition of $U_N(A)$}
In \cite{Kelmer} Kelmer showed\footnote{A similar formula for the
p-adic metaplectique representation was shown already in \cite{MVW}}
that the following can be taken as a definition for $U_N(B)$
\begin{equation}\label{eq:Kelmer Quantization}
U_N(B)=\frac1{N |\ker_N(B-I)|^{1/2}}
\sum_{n\in(\bZ/N\bZ)^2}e(\frac{\omega(n,n B)}{2N})T_N(n(I-B))
\end{equation}
where $\ker_N(B-I)$ denotes the kernel of the map $B-I$ on
$\bZ^2/N\bZ^2$. We take this as the definition of $U_N(A)$ in this
paper.
\begin{prop}\label{prop:U(A) properties}
Let $A\in SL_2(\bZ)$, $A\equiv I\pmod2$. for any $B\in\CA(N)$ let
$U_N(B)$ be as in (\ref{eq:Kelmer Quantization}), then
\begin{enumerate}
\item
For any $I\ne B\in\CA(N)$ $\tr(U_N(B))=1$
\item
Denote by
$$
\eps=\begin{cases} 1 & A \;{\rm{inert}}\\
-1 & A \;{\rm{split}}
\end{cases}
$$
and $\widetilde{U}_N(B)=\eps U_N(B)$, then for $B_1,B_2\in\CA(N)$
\begin{equation}\label{eq:U(A) multiplicativity}
\widetilde U_N(B_1)\widetilde U_N(B_2)=\widetilde U_N(B_1B_2)
\end{equation}
\end{enumerate}
\end{prop}
\begin{proof}
$(1)$ is an immediate result of the fact that
$$\tr T_N(n)=\begin{cases}
N & n=0\pmod N\\
0 & {\rm{otherwise}}
\end{cases}
$$
and the fact that if $I\ne B\in\CA(N)$ then $I-B$ is invertible. For
(2) we recall that $U_N(A)$ is unique up to scalar multiplication,
and there exist a choice $\widetilde U_N(A)$ that is multiplicative.
In particular it was shown in \cite{KR_Hecke},\cite{KR-imrn} that
the eigenvalues of $\CA(N)$ are characters of this group, and that
they are all multiplicity free except the quadratic character, that
in the inert case doesn't appear and in the split case appears with
multiplicity 2. Therefore there exist a multiplicative choice of
phase for which $\tr(\widetilde U_N(B))=\eps\chi_2(B)$. Since
$\chi_2$ is multiplicative we get that $\eps\widetilde
U_N(B)=U_N(B)$ is still multiplicative.
\end{proof}
\begin{remark}
From (\ref{eq:U(A) multiplicativity}) we get that
\begin{equation}\label{eq:power moultiplicativity}
U_N(A^t)=\eps^{t-1}U_N(A)^t
\end{equation}
\end{remark}

%specific case here (multiplicativity for powers of
%$A$), KR results from hecke and Annals, including definition of
%$Q(n),f^\sharp(\nu)$
\subsection{Fluctuations in short windows}
We recall in this section the basic setting from \cite{KRR}. Denote
by $h(t)=\dsone_{[-\frac12,\frac12]}(t)$ the characteristic function
of the interval $[-\frac12,\frac12]$. Set
$$
h_L(x):=\sum_{k\in \bZ} h(L(x-k))
$$
which is then a periodic function, localized on the scale of $1/L$,
and  $\int_0^1 h_L(\theta)^2d\theta =1/L$. The Fourier expansion of
$h_L$ is (in $L^2$ sense)
$$
h_L(x)= \frac 1L\sum_{t\in \bZ} \^h \left(\frac tL \right) e(tx)\;.
$$
where $\^h(y) = \int_{-\infty}^{\infty} h(x) e(-xy) \, dx$.

Let $N$ be a prime which does not divide $\disc(Q)= (\tr A)^2-4$.
Let
$$
P(\theta) := \sum_j h_L(\theta-\theta_j) \langle
\OPN(f)\psi_j,\psi_j \rangle
$$
which is a  sum of matrix elements on a window of size $1/L$ around
$\theta$. Then, in $L^2$ sense, and with $U=U_N(A)$, we have
\begin{equation}
P(\theta)  = \frac 1L\sum_{t\in \bZ} e(t\theta)\^h\left(\frac
tL\right) \tr \{\OPN(f)U^{-t}\} \;.
\end{equation}

In \cite{KRR} we proved the following results about $\tr
\{\OPN(f)U^{-t}\}$ (Lemma 2.3)
\begin{lem}\label{Kelmer's lemma}
Let  $A\in SL_2(\bZ)$ be hyperbolic, and assume that $A\equiv I\mod
2$. Then for any prime $N$ not dividing $\disc(Q)$ and integer $t$
such  that $A^t\neq I\mod N$, we have
\begin{equation}\label{kelmer's formula}
\tr \{ T_N(k)U_N(A^t) \} = (-1)^{k_1k_2} e
\left(\frac{\overline{2}q(k;A^t)}{N} \right)
\end{equation}
and in particular
\begin{equation*}
\tr \{ \OPN(f)U_N(A^t) \} = \sum_k (-1)^{k_1k_2} \^f(k) e
\left(\frac{\overline{2}q(k;A^t)}{N} \right)
\end{equation*}
%where $e(\phi(t))$ is a phase factor that comes from a choice of
%quantization of $A^t$,
where  $\overline{2}$ is the inverse of $2\mod N$.
%and $n=n(t,k)$ is the unique solution modulo $N$ of
%$k\equiv n(A^t-I)\mod N$.
\end{lem}
Following this lemma we denote by
\begin{equation}\label{def:P_nu}
P_\nu(\theta):= \sum_j h_L(\theta-\theta_j) \langle
T_N(n)\psi_j,\psi_j \rangle
\end{equation}
where $\omega(n,nA)=\nu$ (this is well defined by lemma
\ref{Kelmer's lemma}), and since
$|\omega(n,nA)|\leq\|n\|^2_2\|A\|^2_2$ we get the following
decomposition of $P(\theta)$
\begin{cor}\label{cor:sharp expansion}
Let $f(x)=\sum_{\|n\|<R}\hat f(n)e(nx)$ be a trigonometrical
polynomial, then for $N>\|A\|_2R^2$
$$
P(\theta)=\sum_{|\nu|<\|A\|_2R^2}f^\sharp(\nu)P_\nu(\theta)
$$
\end{cor}
%Definition of $P(\theta)$, split into $P_\nu(\theta)$, properties of
%$Q(n)$.
\section{Background on exponential sums}\label{sec:bground exp sums}
 We give some properties of exponential sums. For a given
 algebraic variety $V$ over $k=\bF_p$, and a given rational functions $f,g_1,\dots,g_r$
 on $V$ defined over $k$, we denote $V(k)$ to be the k-rational points
 on $V$, and for multiplicative characters $\chi_1,\dots,\chi_r$ and an additive
 character $\psi$, we define
 \begin{equation}\label{def:gen_rat_exp_sum}
S=S(V,f)=\sum_{x\in V(k)}\psi(f(x))\prod_{i=1}^r\chi(g_i(x))
\end{equation}
and more generally, for any extension $k_n$ of degree $n$ of $k$, we
define
 \begin{equation}\label{def:gen_rat_n_exp_sum}
S_n=S(V\otimes_kk_n,f)=\sum_{x\in
V(k_n)}\psi(\tr(f(x))\prod_{i=1}^r\chi(N(g_i(x)))
\end{equation}
where
\begin{eqnarray*}
N:k_n^* &\to & k^*\\
x & \mapsto  & x\cdot x^p\dots x^{p^{n-1}}\\
\tr:k_n & \to & k\\
x & \mapsto & x+x^p+\cdots+x^{p^{n-1}}
\end{eqnarray*}
are the norm and trace maps respectively. All these sums are
packaged in the corresponding L - function:
\begin{equation}\label{def:L_function}
L(S,T)=\exp\left(\sum_{n=1}^\infty\frac{S_n}nT^n\right)
\end{equation}

\subsection{Deligne's results}
The following was proven by Deligne in \cite{Deligne-WeilII}:
\begin{enumerate}
\item
$L(S,T)$ is a rational function.
\item
The exponential sums $S_n$ satisfy:
\begin{equation}\label{eq:exp sum is trace}
S_n=\sum_i\alpha_i^n-\sum_i\beta_i^n
\end{equation}
 where $\alpha_i$ are the
inverse of the zeros of $L(S,T)$, and $\beta_i$ are inverse of its
poles, and both are called the roots of the exponential sum.
\item
The roots are algebraic integers.
\item All conjugates of a root have the same absolute value which is a positive integer power
of $\sqrt p$.
\end{enumerate}
It was proved by Katz in \cite{Katz-2001}, that there exists a
constant $C$, independent of $p$, such that for a given exponential
sum of type \eqref{def:gen_rat_exp_sum} there are at most $C$ roots.
\subsection{Weil's results}\label{subsec:weil bound}
For a 1 dimensional exponential sum, there is no need for the full
power of Deligne's work, but rather the proof of Weil for RH over
finite fields. We state below the main results concerning this
paper.
\begin{enumerate}
\item
Let $\bF$ be finite field of $q$ elements, and  $\bF[x]$ the ring of
polynomials over $\bF$. For a polynomial $Q(x)\in\bF[x]$, and a
multiplicative character modulo $Q$, we define the corresponding L -
function
$$
L(u,\chi)=\prod_{P\not|Q}\left(1-\chi(P)u^{\deg P}\right)^{-1}
$$
where the product is over all irreducible monic polynomials in
$\bF[x]$. By unique factorization in $\bF[x]$, we have that
$$
L(u,\chi)=\sum_{f\ne0}\chi(f)u^{deg f}=\sum_{n=1}^\infty
a_n(\chi)u^n
$$
where the sum is over all monic polynomials in $\bF[x]$, and
$a_n(\chi)=\sum_{\deg f=n}\chi(f)$.
\item\label{thm:weil bound}
For a nontrivial character modulo $Q$, $L(u,\chi)$ is in fact a
polynomial in $u$, of degree at most $\deg Q-1$.  We may factor it
as follows
$$
L(u,\chi)=\prod_{j=1}^{\deg Q-1}\left(1-\alpha_j(\chi)u\right)
$$
and it was shown by Weil \cite{Weil1948}, that for all
$j=1,\dots,\deg L(u,\chi)$, $|\alpha_j(\chi)|\leq\sqrt q$. Note also
that $a_1(\chi)=-\sum_{j=1}^{\deg Q-1}\alpha_j(\chi)$.
\end{enumerate}
\subsection{Bound for double exponential sums}\label{subsec:variance->individual}
In this part we prove a lemma that gives a sufficient condition for
an exponential sum to have square root cancelation.

We first prove the following proposition that appears previously in
\cite{Bom78,IKBook}.
\begin{prop}\label{prop:limsup}
Let $\xi_1,\dots,\xi_n\in\bC$ be distinct complex numbers of
absolute value one, and $b_1,\dots,b_n\in\bC$ complex numbers. Then
$$
\limsup_{\nu\to\infty}|\sum_{i=1}^nb_1\xi_i^\nu|\geq\left(\sum_{i=1}^n|b_i|^2\right)^{1/2}
$$
\end{prop}
\begin{proof}
For $N\in\bN$ compute the average over $\nu=1,\dots,N$
\begin{equation}\label{eq:average sum}
\frac1N\sum_{\nu=0}^{N-1}|\sum_{i=1}^nb_i\xi_i|^2=\sum_{i=1}^n|b_i|^2+\frac1N\sum_{i\ne
j}b_i\overline{b_j}\frac{1-(\xi_i\xi_j^{-1})^N}{1-(\xi_i\xi_j^{-1})}=\sum_{i=1}^n|b_i|^2+O(\frac1N)
\end{equation}
Now assume that
$$\limsup_{\nu\to\infty}|\sum_{i=1}^nb_i\xi_i^\nu|^2<\sum_{i=1}^n|b_i|^2-\delta$$
for some $\delta>0$. Then in particular the bound is true for $\nu$
large enough, which contradicts \eqref{eq:average sum}.
\end{proof}
The next lemma shows that for general exponential sums, given a
bound on the sum of squares can can lead to a bound on individuals.
\begin{lem}\label{lem:exp sum:variance->individual}
Let $k$ be a finite field of characteristic $char k=p$, and $V$ be
an algebraic variety over $\overline k$ of dimension $N$ and degree
$d$. Let $\underline{\chi}=\{\chi_1,\dots,\chi_l\}$ be
muultiplicative characters of $k^*$, and $\psi$ an and additive
character of $k$, $g_1(x),\dots,g_l(x),f(x)$ rational functions over
$V$. Denote
$$S(\chi,\psi;\underline g,f):=\sum_{x\in
V(k)}\prod_{i=1}^l\chi_i(g_i(x))\psi(f(x))
$$
Assume that there exists $b\in\bN$ and $M\in\bR$ such that for all
$\nu\in\bN$
\begin{equation}\label{eq:lem:exp sum:variance->individual}
\frac1{|k|^\nu}\sum_{0\ne a\in
k_\nu}\left|S_\nu(\underline\chi,\psi_a;\underline g,f)\right|^2\leq
M|k|^{\nu b}
\end{equation}
where $\psi_a(x)=\psi(\tr(ax))$, then there exists a constant
$B=B(N,d,f,g)$ such that $|S(\underline\chi,\psi;\underline
g,f)|\leq B|k|^{b/2}$
\end{lem}

\begin{proof}
By Deligne's result, it suffice to show that all the roots
$\omega_{i}$ of (the L-function of) the exponential sum
$S(\chi,\psi_a)$ are of absolute value $|\omega_i|\leq |k|^{b/2}$.
Denote $r_{max}=\max_{w_i{\rm{\;roots\;of\; S
}}}\{r:|w_i|=|k|^{r/2}\}$, and assume $r_{max}>b$ so there exist
$\omega_1,\dots,\omega_n$ roots of $S(\underline\chi,\psi;\underline
g,f)$, of absolute value $|k|^{r_{max}/2}$, and multiplicities
$\lambda_1,\dots,\lambda_n$, then by proposition \ref{prop:limsup}
with $\xi_i=w_i/|k|^{r_{max}/2}, b_i=\lambda_i$ we get
$$\limsup_{\nu\to\infty}\left|\frac{\sum_{i=1}^n\lambda_i\omega_i}{|k|^{r_{max}\nu/2}}\right|>0$$
and therefore there exist infinitely many $\nu_j$ such that
$|S_{\nu_j}(\underline\chi,\psi;\underline g,f)|\gtrsim
|k|^{r_{max}\nu_j/2}$. for $0\ne c\in\{1,\dots,p-1\}$, let
$\sigma_c\in Gal(\bQ)$ that sends $e(1/p)\mapsto e(c/p)$. Since the
fields $\bQ(e(1/p)),\bQ(e(1/(p-1)))$ are linearly disjoint we get
that $S(\underline\chi,\psi;\underline
g,f)^\sigma=S(\underline\chi,\psi_c;\underline g,f)$, and therefore

\begin{equation*}
\sum_{0\ne  a\in k_{\nu_j}}|S_\nu(\underline\chi,\psi_a;\underline
g,f)|^2 \geq\sum_{0\ne a\in
k}\left|S_\nu(\underline\chi,\psi_a;\underline g,f)\right|^2\gtrsim
|k|^{r_{max}\nu_j+1}
\end{equation*}
which contradicts \eqref{eq:lem:exp sum:variance->individual}
\end{proof}
For exponential sums over a 2 dimensional variety the following
theorem gives a sufficient condition for square root cancelation.

\begin{thm}\label{thm:bound_exp_sum}
Let $V$ be an irreducible algebraic variety over a finite field $k$
of dimension 2, and degree $\delta$. Let $f$ be rational function on
$V$. Suppose that there exists $R$ such that
$$\#\{C\in\overline
k:{\rm{the\; fiber\;f=C\;is\;geometrically\;reducible}}\}<R$$ and
that the degree of all irreducible fibers is at most $d$. Let $\psi$
be an additive character of $k$. Then there exists $B$ such that
\begin{equation}\label{eq:additive square root cancel}
|\sum_{x\in V(k)}\psi(f(x))|\leq B|k|
\end{equation}
\end{thm}

\begin{proof}
By lemma \ref{lem:exp sum:variance->individual} it suffices to show
that there exists $M$ such that
$$
\frac1{|k|^\nu}\sum_{0\ne a\in k_\nu}|S(\psi_a)|^2\leq M|k|^{2\nu}
$$
which is equivalent to show that
\begin{equation}\label{sum:complete sum variance}
\sum_{a\in k_\nu}|S(\psi_a)|^2=N_\nu^2+O(|k|^{3\nu})
\end{equation}
where $N_\nu=\#\{\vec x\in V(k_\nu)\}=|k|^{2\nu}+O(|k|^{3\nu/2})$ by
irreducibility of $V$ and Lang-Weil theorem \cite{LaWei}. Writing
the sum in (\ref{sum:complete sum variance}) explicitly we get
\begin{eqnarray*}
\sum_{a\in k_\nu}\sum_{x\in V(k_\nu)}\sum_{x'\in V(k_\nu)}\psi(a(f(x)-f(x')))=\\
|k|^\nu\sum_{x\in V(k_\nu)}\sum_{x'\in
V(k_\nu)}\sum_{f(x)=f(x')}1=|k|^\nu \sum_{C\in
k_\nu}|f_\nu^{-1}(C)|^2
\end{eqnarray*}
Where $f_\nu^{-1}(C)=\{x\in V(k_\nu):f(x)=C\}$. The number of points
on $f_\nu^{-1}(C)$ is given by
\begin{equation}
|f_\nu^{-1}(C)|=\frac1{|k|^\nu}\sum_{a\in k_\nu}\sum_{\vec x\in
V(k_\nu)}\psi(a(f(\vec x)-C))=\frac{N_\nu}{|k|^\nu}+E_\nu(C)
\end{equation}
where
\begin{equation}\label{def:remainder of sum}
E_\nu(C)=\frac1{|k|^\nu}\sum_{0\ne a\in k_\nu}\sum_{\vec x\in
V(k_\nu)}\psi(a(f(\vec x)-C))
\end{equation}
and therefore
\begin{eqnarray*}
|k|^\nu \sum_{C\in k_\nu}|f_\nu^{-1}(C)|^2=|k|^\nu \sum_{C\in
k_\nu}\left(\frac{N_\nu}{|k|^\nu}+E_\nu(C)\right)^2=\\
N_\nu^2+\sum_{C\in k_\nu}\left(2N_\nu E_\nu(C)+|k|^\nu
E_\nu(C)^2\right)
\end{eqnarray*}
By the assumption on the fibers and the Riemann hypothesis for
curves we get that $E_\nu(C)=O(|k|^{\nu/2})$ for all $C$ except at
most $R$. From \eqref{def:remainder of sum} we have that $\sum_{C\in
k_\nu}\tilde E_\nu(C)=0$ and therefore we get that
$$
\sum_{a\in k_\nu}|S(\psi)|^2=N_\nu^2+O(|k|^{3\nu})
$$
which concludes the proof.
\end{proof}
\begin{remark}
As was seen throughout the proof the irreducibility assumptions can
be replaced by cardinality assumptions on $V$ and the fibers, that
is if $\# V=|k|^2+O(|k|^{3/2})$ and the fibers satisfy
$|f_\nu^{-1}(C)-|k|^\nu|\leq B\sqrt|k|^\nu$ for an absolute constant
$B$ then the theorem holds as well.
\end{remark}
\subsection{Bounds for character sums over $\bF_q$}
We prove here a condition for a square root cancelation for one
dimensional sums involving many multiplicative characters. We prove
the following
\begin{thm}\label{thm:one_dim_bound}
Let $k=\bF_q$ be the field with $q=p^n$ elements ($char(k)=p$), and
let $\chi_1,\dots,\chi_m$ be nontrivial multiplicative characters of
$k$. Let $P_1(x),\dots,P_m(x)\in k[x]$ be monic irreducible
polynomials of degrees $d_1,\dots,d_m$ respectively. Then
\begin{equation}\label{eq:one_dim_bound}
\sum_{t\in\bF_q}\prod_{i=1}^m\chi_i(P_i(t))\leq\left(\sum_{i=1}^md_i-1\right)\sqrt
q
\end{equation}
\end{thm}
To prove this bound we construct a polynomial $Q_\chi(x)\in k[x]$ of
degree less then $\sum d_i$, and a nontrivial character
$\nu_\chi:\left(k[x]/Q_\chi(x)\right)^x\to\bC$, such that
$\nu_\chi(x-t)=\prod_{i=1}^m\chi_i(P(t))$.
\begin{prop}\label{prop:poly constr}
Let $k$ be a field, and let $\{x_1,\dots,x_l\}\subset \bar k$ be
finite set invariant under Galois action. Then for any set
$y=\{y_1,\dots,y_l\}\subset\bar k$ invariant under Galois action,
there exists a unique monic polynomial $P_y(x)\in k[x]$ of degree
$l$, such that $P_y(x_i)=a_i,\forall\sigma\in
Gal(k)\;P_y(\sigma(x_i))=\sigma(a_i)$.
\end{prop}
\begin{proof}
The existence and uniqueness of $P(x)\in\overline k[x]$ is a
standard linear algebra argument. To show that $P(x)\in k[x]$ we
notice that for any $\sigma\in Gal(k)$
$$
\sigma(P)(x_i)=\sigma(P(\sigma^{-1}(x_i)))=\sigma(\sigma^{-1}(P(x_i)))=P(x_i)
$$
and by uniqueness of $P(x)$ we get that $\sigma(P)(x)=P(x)$ and
hence $P(x)\in k[x]$.
\end{proof}
Proposition \ref{prop:poly constr} will give us a way to construct
the required $\nu_\chi$. We do it using the resultant of two
polynomials
\begin{defn}
Let $P,Q\in k[x]$ be two monic polynomials. Define
$$
res(P,Q):=\prod_{\substack{x_i {\rm{roots\;of\;}}Q\\y_j
{\rm{roots\;of\;}}P}}(x_i-y_j)=\prod_{y_j
{\rm{roots\;of\;P}}}Q(y_j)=(-1)^{deg(P)}\prod_{x_i
{\rm{roots\;of\;Q}}}Q(x_i)
$$
\end{defn}
\begin{cor}
Let $P_1,\dots,P_m\in k[x]$ be distinct monic irreducible
polynomials of degrees $d_1,\dots,d_m$ respectively. Then for any
$a=(a_1,\dots,a_m)\in k^m$ there exists a polynomial $Q_a(x)\in
k[x]$ such that $res(P_i,Q_a)=a_i$
\end{cor}
\begin{proof}
For any $P_i$ let $Y_i=\{y_{ij}\}_{j=1}^{d_i}\subset\overline k$ be
the set of its roots in the algebraic closure. Then $Y=\cup_iY_i$ is
an invariant set under the Galois group action. For $a_i$ let
$Z_i=\{z_{i1},\dots,z_{id_i}\}$ be a set of Galois conjugates
elements in $k_i:=k(y_{i1},\dots,y_{id_i})$ such that
$N_{k_i/k}(z_i)=z_1\cdots z_{d_i}=a_i$, and let $Z=\cup_{i}Z_i$. By
proposition \ref{prop:poly constr} there exists a polynomial
$Q_a(x)\in k[x]$ such that $Q_a(y_{ij})=z_{ij}$. Then
$$
res(P_i,Q_a)=\prod_{y_{ij}
{\rm{roots\;of\;}}P}Q_a(y_{ij})=z_{i1}\cdots z_{id_i}=a_i
$$
\end{proof}
We can now conclude the proof of theorem \ref{thm:one_dim_bound}.
Denote by $Q(x)=lcm(P_1,\dots,P_m)\in k[x]$, and by $\nu_\chi$ the
character of $\left(k[x]/Q(x)\right)^x$ defined by
$$
\nu_\chi(F)=\prod_{i=1}^m\chi_i(res(P_i,F))
$$
By previous corollary $\nu_\chi$ is nontrivial, and by definition of
the resultant it is well defined modulo $Q$. Thus by Weil's result
the theorem is proved.
\section{A family of exponential sums}\label{sec:Fam_exp_sum}
Let $k$ be a finite field of $q=p^n$ elements. For $\psi$ be an
additive character of $k$ and $\chi$ a multiplicative character of
$k^*$ We define the following exponential sum
\begin{equation}\label{def:hecke exp sum}
F(\chi;\psi)=\sum_{0,1\ne x}\chi(x)\psi(\frac{1+x}{1-x})
\end{equation}
We consider the family $\{F(\chi;\psi)\}_\chi$ where $\chi$ runs
through all characters of $k^*$. It was shown in \cite{KRR} that
\begin{equation}\label{eq:rate bound}
|F(\chi;\psi)|\leq\sqrt q
\end{equation}
In light of this result we normalize the sum and define
$$
\tilde F(\chi;\psi)=\frac{F(\chi;\psi)}{\sqrt q}
$$
and consider the family $\{\tilde F(\chi;\psi)\}_{\chi,\psi}$. The
following proposition give some basic properties of this family
\begin{prop}\label{prop:basic prop exp sum}
Let $F(\chi;\psi)$ be as above.
\begin{enumerate}
\item
For any pair $\chi,\psi$ as above, the sum $F(\chi;\psi)$ is r a
real number.
\item
$$
\frac1{q-1}\sum_{\chi}\tilde F(\chi;\psi)=0
$$
\item
$$\frac1{q-1}\sum_\chi|\tilde F(\chi;\psi)|^2=1-\frac2q$$
\item
For any $\chi_1\ne\chi_2$, and $\psi_1,\psi_2$
\begin{equation}\label{eq:exp_sum_cov}
\frac1{q-1}\sum_{\chi}\tilde F(\chi_1\chi;\psi_1)\overline{\tilde
F(\chi_2\chi;\psi_2)}\ll\frac1{\sqrt q}
\end{equation}
\end{enumerate}
\end{prop}
\begin{proof}
The first part of the proposition follows from the simple
observation that
$$
\overline {F(\chi;\psi)}=\sum_{0,1\ne
x}\chi(x^{-1})\psi(-\frac{1+x}{1-x})=\sum_{0,1\ne
x}\chi(x^{-1})\psi(\frac{1+x^{-1}}{1-x^{-1}})=F(\chi;\psi)
$$
Parts 2,3 of the proposition are immediate consequence of the
orthogonality relations of characters. For \eqref{eq:exp_sum_cov} we
have
\begin{equation}
\frac1{q-1}\sum_{\chi}\frac{F(\chi_1\chi;\psi_1)}{\sqrt
q}\overline{\frac{F(\chi_2\chi;\psi_2)}{\sqrt
q}}=\frac1q\sum_{0,1\ne x}
\chi_1\overline{\chi_2}(x)\psi_1\overline{\psi_2}(\frac{1+x}{1-x})\ll\frac1{\sqrt
q}
\end{equation}
where
\begin{eqnarray*}
\chi_1\overline\chi_2(x)=\chi_1(x)\chi_2(x^{-1})\\
\psi_1\overline\psi_2(y)=\psi_1(y)\psi_2(-y)
\end{eqnarray*}
ans the last inequality is due to \eqref{eq:rate bound}.
\end{proof}
The last proposition can be considered as computation of mean and
variance for fixed $\psi$ and running over $\chi$, and the third
result as covariance for two random variables $\tilde
F(\chi_1\chi;\psi),\tilde F(\chi_2\chi;\psi)$ running over $\chi$,
and in fact proves that for any two additive characters
$\psi_1,\psi_2$ the random variables $\tilde F(\chi;\psi_1), \tilde
F(\chi;\psi_2)$ become uncorrelated. The following conjecture
suggests even a stronger behaviour.
\begin{conj}\label{conj:exponential_sum}
Let $\chi,\psi,F(\chi;\psi)$ be as defined above, then as
$q\to\infty$ through primes, we have the following:
\begin{enumerate}
\item
 The sets
$\{\frac{F(\chi;\psi)}{\sqrt N}\}_\chi$ become equidistributed with
respect to the Sato-Tate distribution $\mu_{ST}$, that is the
distribution of $\tr(U)$ where $U\in SU(2)$ is random matrix with
respect to Haar measure.
\item
For any finite field $k$, let $(\chi_i,\psi_i)_{i=1}^m$ be a set of
$m$ distinct pairs of multiplicative and additive characters of $k$.
Then the sets
$$\{(\frac{\tilde F(\chi_1\chi;\psi_1)}{\sqrt q},\frac{\tilde F(\chi_m,\psi_m)}{\sqrt q})\}_\chi$$
become equidistributed with respect to the product of $m$ Sato-Tate
measures, that is they become independent.
\item
In particular, the mixed moments of $m$ distinct pairs
$$(\tilde F(\chi_i;\psi_i),\dots,\tilde F(\chi_m;\psi_m))$$
satisfy
\begin{equation}\label{eq:moments_of_exp_sum}
\frac1{|q-1|}\sum_{\chi}\prod_{i=1}^m\left(\tilde
F(\psi_i;\chi\chi_i)\right)^{e_i}=\mathbb
E(\prod_{i=1}^mX_{\psi_i,\chi_i}^{e_i})+O(\frac1{\sqrt q})
\end{equation}
where $X_{\psi_i,\chi_i}$ $i=1,\dots,m$ are IID random variables
with Sato-Tate distribution
\end{enumerate}
\end{conj}
In the following sections we give some agreement with this
conjecture by proving the following theorems
\begin{thm}\label{thm:exp_sum_3rd_moment}
Let $k$ be a finite field with $q=p^n$ elements, and let
$\chi_1,\chi_2,\chi_3$ be any 3 multiplicative characters of $k^*$,
and $\psi_1,\psi_2,\psi_3$ be any nontrivial additive characters of
$k$, then
\begin{equation}\label{eq:exp_sum_3rd_mmnt}
\frac1{q-1}\sum_{\chi}\tilde F(\chi_1\chi;\psi_1)\tilde
F(\chi_2\chi;\psi_2)\tilde F(\chi_3\chi;\psi_3)\ll\frac1{\sqrt q}
\end{equation}
\end{thm}
\begin{thm}\label{thm:exp_sum_5th_moment}
Let $k$ be a finite field with $q=p^n$ elements, let $\chi_1$ be any
multiplicative character of $k^*$, and $\psi_1,\dots,\psi_5$ be any
nontrivial additive characters of $k$, then
\begin{equation}\label{eq:exp_sum_5th_mmnt}
\frac1{q-1}\sum_{\chi}\prod_{i=1}^5\tilde
F(\chi_1\chi;\psi_i)\ll\frac1{\sqrt q}
\end{equation}
\end{thm}
\begin{remark}
In section \ref{sec:discussion} we show numerical evidence for
conjecture \ref{conj:exponential_sum}.
\end{remark}
\section{Proof of theorem
\ref{thm:exp_sum_3rd_moment}}\label{sec:prf_exp_sum_3rd_mmnt}
We
start by making a change of variables in the sum over $\chi$ by
letting $\chi\mapsto\overline{\chi}_3\chi$, and we therefore may
assume that $\chi_3$ is trivial. Moreover since
$\psi_1,\psi_2,\psi_3$ are nontrivial, there exists $0\ne A_0,A,B\in
k$ such that
$\psi_1(x)=\psi(A_0x),\psi_2(x)=\psi(Ax),\psi_3(x)=\psi(Bx)$ where
$\psi$ is a generator of the group of additive characters. We next
sum over $\chi$ to get
\begin{equation}\label{eq:exp_sum_3rd_mmnt_explct}
S(\chi_1,\chi_2;\psi):=\sum_{x,y}\chi_1(x)\chi_2(y)\psi\left(A_0\frac{1+x}{1-x}+A\frac{1+y}{1-y}+B\frac{xy+1}{xy-1}\right)
\end{equation}
(without loss of generality we assume $\psi=\psi_1$ and therefore
$A_0=1$). Under the change of variables
$(x,y)\mapsto(\frac{1-x}{1+x},\frac{1-y}{1+y})$, the sum changes to
\begin{equation}\label{eq:exp_sum_3rd_mmnt_explct2}
S(\chi_1,\chi_2;\psi)=\sum_{x,y,xy\ne0,-1}\chi_1\left(\frac{1-x}{1+x}\right)\chi_2\left(\frac{1-y}{1+y}\right)\psi\left(x+Ay-B\frac{xy+1}{x+y}\right)
\end{equation}
\begin{prop}\label{prop:3rdmmnt_fibers_irred}
Let $k$ be a finite field. For $0\ne A,B\in k$ let
$f(x,y)=x+Ay-B\frac{xy+1}{x+y}$. Then for all $C\in\overline k$
satisfying $C^2\ne(A-B-1)^2-4B$ the fiber $f(x,y)=C$ is absolutely
irreducible
\end{prop}
\begin{proof}
For $C\in\overline k$ consider the equation
$$
f(x,y)=x+Ay-B\frac{xy+1}{x+y}=C
$$
multiplying it by $x+y$ this turns out to be
$$
x^2+(A-B+1)xy+y^2-Cx-Cy-B=0
$$
which is a quadratic curve. For a quadratic curve
$a_{11}x^2+2a_{12}xy+a_{22}y^2+2b_1x+2b_2y+c$ it is known that it is
absolutely irreducible over $k$ if the determinant
\begin{equation*}
\left|\begin{matrix} a_{11} & a_{12} & b_1\\
a_{12} & a_{22} & b_2\\
b_1 & b_2 & c
\end{matrix}\right|\ne0
\end{equation*}
In our case it is $\frac B4\left(C^2-(A^2+B^2-2A-2B-2AB+1)\right)$,
since $B\ne0$  we get that for at most 2 values of $C$ the curve is
reducible. Furthermore, since it is an irreducible quadratic curve,
its intersection with the curve $x+y=0$ has at most $2$ affine
points points and therefore multiplying by that factor was valid.
\end{proof}
\begin{cor}\label{cor:bound_for_trivial_case}
Let $\chi_1$ be any multiplicative character of $k$. Then there
exists $0<M\in\bR$ such that
\begin{equation}
\frac1{q-1}\sum_{\chi}\tilde F(\chi_1\chi;\psi_1)\tilde
F(\chi_1\chi;\psi_2)\tilde F(\chi_1\chi;\psi_3)\leq M\frac1{\sqrt q}
\end{equation}
\end{cor}
\begin{proof}
This is an immediate corollary of theorem \ref{thm:bound_exp_sum}
\end{proof}
For the cases were not all characters are equal we use the following
proposition that observes some geometric properties of the fibers.
\begin{prop}\label{prop:geometric_restrictions}
Let $k$, $f(x,y)$ be as above, and $C\in\overline k$ satisfying
$C^2\ne(A-B-1)^2-4B,0$. Denote $f_C=\{(x,y)\in k^2:f(x,y)=C\}$. Then
$f_C$ satisfies
\begin{enumerate}
\item\label{prop:setting not const}
For any $a\in k$ the intersection of the curve $x-a$ or $y-a$ with
$f_C$ has at most 2 points.
\item\label{prop:setting not line}
For any $a\in k$, the intersection of the curve $\frac{x-1}{x+1}=
a\frac{y-1}{y+1}$ with $f_C$ has at most 4 points.
\item\label{prop:setting not inverse}
For any $a\in k$, the intersection of the curve $\frac{x-1}{x+1}=
a\frac{y+1}{y-1}$ with $f_C$ has at most 4 points.
\item\label{prop:setting not quadratic}
The intersection of the curve $(x^2-1)(y^2-1)=0$ with $f_C$ has at
most 8 points.
\item\label{prop:setting not square}
If $f_C$ is not empty, neither $\frac{x+1}{x-1}$ nor
$\frac{y+1}{y-1}$ are contained in any multiplicative coset of the
subgroup of squares of $k^*$ along $f_C$.
\end{enumerate}
\end{prop}
\begin{proof}
By assumption on $C$, and by proposition
\ref{prop:3rdmmnt_fibers_irred} we have that $f_C$ is an irreducible
quadratic curve.
%, and so by the chord-arc method we get that it is
%parameterizable as desired.
As such, if $G$ is any other curve (not necessarily irreducible) not
containing $f_C$, the number of points $\sharp G\cap f_C$ is bounded
by the product $deg(G)\cdot deg(f_C)=2deg(G)$. Therefore, since
$x-a,y-a,(1-a)(xy-1)-(1+a)(x-y),(1-a)(xy+1)+(1+a)(x+y)$,
$(x^2-1)(y^2-1)$ are all coprime to $f(x,y)-C$ we get properties
\ref{prop:setting not const}, \ref{prop:setting not line},
\ref{prop:setting not inverse}, \ref{prop:setting not quadratic}.
For property \ref{prop:setting not square} we give a parametrization
of $f_C$. Choose a point $(x_1,y_1)\in f_C$ such that $x_1\ne0,\pm1$
(such a choice is possible since $f_C$ is not empty and by property
\ref{prop:setting not const}). Then the following is a
parametrization of $f_C$:
\begin{eqnarray}\label{eq:curve_parametrization}
x(t)=\frac{x_1(At^2-Ct-B)}{At^2-((A-B+1)x_1+2Ay_1)t+C(x_1+y_1)+B}\\
y(t)=\frac{(C- (A- B+1) x_1 - A y_1)t^2 +t (2 B + Cx_1) - By_1}{A
t^2- t ((A-B+1)x_1- 2 A y_1) + C( x_1 + y_1) +B}
\end{eqnarray}
This gives similar expressions for
$\frac{x-1}{x+1},\frac{y-1}{y+1}$. Denote by $q_1(t),q_2(t)$ the
numerator and denominator for $\frac{x-1}{x+1}$ and $q_3(t),q_4(t)$
the numerator and denominator for $\frac{y-1}{y+1}$. Direct
computation gives that
\begin{eqnarray*}
disc(q_1(t))=disc(q_3(t))=(C-A+B-1)^2x_1^2\\
disc(q_2(t))=disc(q_4(t))=(C+A-B+1)^2x_1^2\\
\end{eqnarray*}
Therefore the only possibility for all to be squares is if
$A-B+1=0,C=0$, by choosing $C\ne0$ we get property \ref{prop:setting
not square}.
\end{proof}
\begin{cor}\label{cor:char_sum_sqrt_cncl}
Let $k$ be a finite field with $q=p^n$ elements, and $f(x,y)$ as
above. Let $\chi_1,\chi_2$ be multiplicative characters of $k$ not
both trivial. For $\nu\in\bN$ be the extension of $k$ of degree
$\nu$, Then for all $C\in k_\nu$ satisfying $C^2\ne(A-B-1)^2-4B,0$
we have
\begin{equation}\label{eq:char_sum_sqrt_cncl}
|\sum_{f(x,y)=C}\chi_1(\frac{1-x}{1+x})\chi_2(\frac{1-y}{1+y})|\leq7\sqrt
q
\end{equation}
\end{cor}
\begin{proof}
By proposition \ref{prop:3rdmmnt_fibers_irred} the sets
$f_C=\{(x,y)\in k^2:f(x,y)=C\}$ are irreducible quadratic curves,
and hence as in proposition \ref{prop:geometric_restrictions} we can
parameterize the curve and the sum becomes
\begin{equation}
\sum_{t\in k}\chi_1(q_1(t))\chi_2(q_2(t))=\sum_{t\in\
k}\chi_1(p_1(t))\overline\chi_1(p_2(t))\chi_2(p_3(t))\overline\chi_2(p_4(t))
\end{equation}
where
$$q_1(t)=\frac{p_1(t)}{p_2(t)},q_2(t)=\frac{p_3(t)}{p_4(t)},i=1,2$$
This can be written in the form
$$
\varepsilon\sum_{t\in k}\prod_{i=1}^m\widetilde\chi_i(\tilde p_i(t))
$$
where $\tilde p_i(t)$ area different monic irreducible polynomials,
$|\varepsilon|=1$, by decomposing $p_i(t),i=1,\dots,4$ into
irreducible parts, joining equal parts together, and powers are
absorbed into the characters. In order to apply theorem
\ref{thm:one_dim_bound}, we must check that there exists a
nontrivial pair $\widetilde\chi_i,\tilde p_i(t)$, that is, there
exists $1\leq i\leq m$, such that $\widetilde\chi_i$ is not the
trivial character, and $\widetilde p_i(t)\not\equiv1$. However
$\widetilde\chi_i,i=1,\dots,4$ are the trivial characters, and
$\tilde p_i(t), i=1,\dots,4$ are constant if complete cancelation
occurs already in
$$
\chi_1(p_1(t))\overline\chi_1(p_2(t))\chi_2(p_3(t))\overline\chi_2(p_4(t))
$$
that is one of the following holds
\begin{enumerate}
\item
$p_1(t)/p_2(t)=Const.$, $p_3(t)/p_4(t)=Const$
\item
$p_1(t)/p_3(t)=Const.$, $p_2(t)/p_4(t)=Const$, and
$\chi_1=\overline\chi_2$
\item
$p_1(t)/p_4(t)=Const.$, $p_2(t)/p_3(t)=Const$, and $\chi_1=\chi_2$
\item
$p_1p_2/(p_3p_4)=Const.$, and $\chi_1^2=\chi_2^2$ are trivial.
\item
$p_1,p_2,p_3,p_4=\square$, and $\chi_1^2=\chi_2^2$ are trivial.
\end{enumerate}
and each of these conditions corresponds to a geometric restriction
that was proved impossible in proposition
\ref{prop:geometric_restrictions}. We can therefore imply theorem
\ref{thm:one_dim_bound}. Since $\deg (p_i(t))\leq 2$, then the sum
of their degrees is at most 8, hence the sum of the degrees of
$\widetilde p_i(t)$ is at most 8 also, and therefore we get the
required result.
\end{proof}

We can now conclude the proof of theorem
\ref{thm:exp_sum_3rd_moment}. Let $\chi_1,\chi_2$ be multiplicative
characters of $k$ not both trivial, we want to show that there
exists $M>0$ such that the following bound holds
\begin{equation}\label{eq:3rd_mmnt_rqurd_bnd}
|S(\chi_1,\chi_2;\psi)|\leq M|k|
\end{equation}
where $S(\chi_1,\chi_2;\psi)$ was defined in
\eqref{eq:exp_sum_3rd_mmnt_explct}.
\begin{prop}\label{prop:3rd_mmnt_variance}
Let $S(\chi_1,\chi_2;\psi)$ be as above. Then for all $\nu\in\bN$
$$
\frac1{|k_\nu|}\sum_{a\in k_\nu}|S(\chi_1\circ N,\chi_2\circ
N;(\psi\circ tr)_a)|^2\leq52|k|^{2\nu}
$$
\end{prop}
\begin{proof}
To show the bound we compute the sum over $a$ to get
\begin{eqnarray*}
&&\frac1{|k_\nu|}\sum_{a\in k_\nu}|S(\chi_1\circ N,\chi_2\circ
N;(\psi\circ \tr)_a)|^2=\\
&&\sum_{C\in k_\nu}|\sum_{x,y\in k_\nu:f(x,y)=C}\chi_1\circ
N(\frac{1-x}{1+x})\chi_2\circ N(\frac{1-y}{1+y})|^2
\end{eqnarray*}
and by corollary \ref{cor:char_sum_sqrt_cncl} the inner sum is
bounded by $49|k_\nu|$ for all but at most 3 values of $C$, hence
the bound is proved.
\end{proof}
\begin{cor}\label{cor:3rd_mmnt_rqurd_bnd}
Let $k$ be a finite field with $char(k)\ne2$. Then there exists
$M>0$ such that for any two multiplicative characters, and $\psi$
additive character of $k$ the sum $S(\chi_1,\chi_2;\psi)$ satisfies
$$
|S(\chi_1,\chi_2;\psi)|\leq M|k|
$$
\end{cor}
\begin{proof}
This is an immediate corollary of proposition
\ref{prop:3rd_mmnt_variance} and lemma \ref{lem:exp
sum:variance->individual}
\end{proof}
\section{Proof of theorem
\ref{thm:exp_sum_5th_moment}}\label{sec:prf_exp_sum_5th_mmnt}
For the proof of theorem \ref{thm:exp_sum_5th_moment} we use an
averaging technique that will allow us to distinguish symmetries of
the sum.
\begin{defn}
Let $k$ be a finite field, $\chi_1,\chi_2$ be multiplicative
characters of $k$, $\psi$ an additive character and $a\in k^*$.
Denote
\begin{eqnarray*}
F(\chi_1,\chi_2;\psi,a)=\sum_{x\in
k}\chi_1(x+a)\overline\chi_2(x-a)\psi(a^{-1}x)\\
G(\chi_1,\chi_2;\psi)=\frac1{|k^*|}\sum_{a\in
k^*}F(\chi_1,\chi_2;\psi,a)
\end{eqnarray*}
\end{defn}
\begin{prop}\label{prop:5th_averaging_settng}
Let $k$ be a finite field, $\chi_1,\chi_2,\psi,a$ as above. Then
\begin{enumerate}
\item\label{prop:av_set_mtrx_elmnt}
$F(\chi_1,\chi_2;\psi,a)=\chi_1\overline\chi_2(a)F(\chi_1,\chi_2;\psi,1)$
\item\label{prop:av_set_avrgd_mtrx_lmnt}
\begin{equation}
G(\chi_1,\chi_2;\psi)=\begin{cases}F(\chi_1;\psi)& \chi_1=\chi_2\\0
& \chi_1\ne\chi_2\end{cases}
\end{equation}
\item
Let $\theta$ be any multiplicative character of $k$, and
$\psi_1,\dots,\psi_5$ any nontrivial additive characters of $k$,
then
\begin{equation}\label{eq:5th_mmnt_avgd_sum}
\frac1{|k^*|}\sum_\chi\prod_{i=1}^5F(\theta\chi;\psi_i)=\frac1{|k^*|^6}\sum_{\chi_1,\dots,\chi_5\in
\widehat{k^*}}\prod_{i=1}^5G(\chi_i,\chi_{i+1};\psi_i)
\end{equation}
where inside the product we consider 5+1 as 1.
\end{enumerate}
\end{prop}
\begin{proof}
Part \ref{prop:av_set_mtrx_elmnt} is immediate under the change of
variables $x\mapsto ax$ which is invertible since $a\in k^*$. Part
\ref{prop:av_set_avrgd_mtrx_lmnt} and the equality in
\eqref{eq:5th_mmnt_avgd_sum} are then immediate from part
\ref{prop:av_set_mtrx_elmnt} and the orthogonality relations.
\end{proof}
\begin{lem}\label{lem:5th_mmnt_is_exp_sum}
Let $\theta$ be a multiplicative character of $k$ ($char(k)\ne2$),
$\psi$ a nontrivial additive character of $k$. for $A_1,\dots,A_5\in
k^*$ let $\psi_i(x)=\psi(A_ix)$ be nontrivial additive characters of
$k$. Denote by
\begin{equation*}
{\bf{V}}(\underline A)=\left\{0\ne
a_1,a_2,a_3,a_4:\substack{A_1a_1+A_2a_2+A_3a_3+A_4a_4+A_5=0\\a_1^{-1}+a_2^{-1}+a_3^{-1}+a_4^{-1}+1=0}\right\}
\end{equation*}
and
\begin{equation*}
\tilde h_A(a_2,a_3,a_4)=\sum_{2\leq
i<j\leq4}\left(\frac{A_ia_i}{a_j}-\frac{A_ja_j}{a_i}\right)+\sum_{i=1}^4A_ia_i-A_5\sum_{i=2}^4a_i^{-1}
\end{equation*}
then
\begin{eqnarray}\label{eq:5th_mmnt is_exp_sum}
&&\frac1{|k^*|}\sum_\chi\prod_{i=1}^5\tilde
F(\theta\chi;\psi_i)=\\
\nonumber&&\frac1{|k|^{3/2}}\sum_{a_1,a_3,a_4,a_5\in
{\bf{V}}(\underline A)}{\hspace{-20pt}}\psi_1(-\tilde
h_A(a_2,a_3,a_4))
\end{eqnarray}
\end{lem}
\begin{proof}
By \eqref{eq:5th_mmnt_avgd_sum} we have that
\begin{equation}\label{eq:5th_mmnt_frst_sum}
\frac1{|k^*|}\sum_\chi\prod_{i=1}^5\tilde
F(\theta\chi;\psi_i)=\frac1{|k^*|^6|k|^{5/2}}\sum_{\chi_1,\dots,\chi_5\in
\widehat{k^*}}\prod_{i=1}^5G(\chi_i,\chi_{i+1};\psi_i)
\end{equation}
Summing over $\chi_1,\dots,\chi_5$ in \eqref{eq:5th_mmnt_frst_sum},
we get
\begin{eqnarray}\label{eq:5th_mmnt_scnd_sum}
&&\frac1{|k^*||k|^{5/2}}\sum_{\substack{a_1,\dots,a_5\\x_1,\dots,x_5}\in
V}\psi(\sum_{i=1}^5A_ia_i^{-1}x_i)=\\
\nonumber&&\frac1{|k^*||k|^{5/2}}\sum_{\substack{a_1,\dots,a_5\in
k^*\\a_1+\cdots+a_5=0}}\sum_{x_1\in
k}\psi((\sum_{i=1}^5A_ia_i^{-1})x_1+h(a_1,\dots,a_5))
\end{eqnarray}
 where
\begin{eqnarray*}
&&V=\left\{\substack{a_1,\dots,a_5\in k^*\\x_1,\dots,x_5\in
k}:\substack{x_i-a_i=x_{i+1}+a_{i+1},i=1,\dots,4,\\x_1+a_1=x_5-a_5}\right\}\\
&&h(a_1,\dots,a_5)=a_1\left(\sum_{i=1}^5\frac{A_i}{a_i}\right)+2\sum_{j=2}^4a_j\left(\sum_{i=j+1}^{5}\frac{A_i}{a_i}\right)
\end{eqnarray*}
Summing over $x_1$ gives
\begin{equation}\label{eq:5th_mmnt_thrd_sum}
\frac1{|k^*||k|^{3/2}}\sum_{\substack{a_1,\dots,a_5\in
k^*\\a_1+\cdots+a_5=0\\\frac{A_1}{a_1}+\cdots+\frac{A_5}{a_5}=0}}{\hspace{-10pt}}\psi(h(a_1,\dots,a_5))=\frac1{|k|^{3/2}}\sum_{a_1,a_3,a_4,a_5\in
{\bf{V}}(\underline A)}{\hspace{-20pt}}\psi(-\tilde
h_A(a_2,a_3,a_4))
\end{equation}
where the last equality is given by the change of variables
$a_i\mapsto a_5^{-1}A_ia_i$, using the resulting equality
$$a_1^{-1}+a_2^{-1}+a_3^{-1}+a_4^{-1}+1=0,$$
and summing over $a_5$ which no longer appears in the sum.
\end{proof}
\begin{prop}\label{prop:5th_mmnt_exp_sum_bnd}
Let $k$ be a finite field, and let $\psi$ be an additive character
of $k$. For $\underline A=(A_1,\dots,A_5)\in k^*$, let
${\bf{V}}(A)$, $\tilde h_A$ as above. Denote by
$$
S_A(\psi)=\sum_{a_1,a_3,a_4,a_5\in {\bf{V}}(\underline
A)}{\hspace{-20pt}}\psi_1(\tilde h_A(a_2,a_3,a_4))
$$
Then there exist $M>0$ such that
$$
S_A(\psi)\leq M|k|
$$
\end{prop}
\begin{proof}
% where
%\begin{eqnarray*}
%&&{\bf{V}}(\underline A)=\left\{0\ne
%a_1,a_3,a_4,a_5:\substack{A_1a_1+A_2+A_3a_3+A_4a_4+A_5a_5=0\\a_1^{-1}+1+a_3^{-1}+a_4^{-1}+a_5^{-1}=0}\right\}\\
%&&\tilde
%h_A(a_3,a_4,a_5)=\frac1{a_3}+\frac1{a_4}+\frac1{a_5}+A_3a_3\left(\frac1{a_4}+\frac1{a_5}\right)+\frac{A_4a_4}{a_5}
%\end{eqnarray*}
 In appendix \ref{ap:irred_proofs} we prove the following
properties of ${\bf V}(A),\tilde h_A$
\begin{enumerate}
\item
For $A=(A_1,\dots,A_5)\in(k^*)^5$ the variety ${\bf V}(A)$ is an
irreducible two dimensional algebraic variety
\item
Except 14 values, all fibers of $\tilde h_A:{\bf{V}}(A)\to\overline
k$ are absolutely irreducible
\end{enumerate}
and therefore the proposition is a corollary of theorem
\ref{thm:bound_exp_sum}.
\end{proof}
Theorem \ref{thm:exp_sum_3rd_moment} is now a corollary of lemma
\ref{lem:5th_mmnt_is_exp_sum} and proposition
\ref{prop:5th_mmnt_exp_sum_bnd}.
\begin{remark}
Using the same averaging trick one can show that the third moment
satisfies
$$
\frac1{p-1}\sum_\chi
F(\chi;\psi_1)F(\chi;\psi_2)F(\chi;\psi_3)=p\sum_{a^2=B}\psi(a)
$$
where $B=(A_1+A_2+A_3)^2-4(A_1A_2+A_1A_3+A_2A_3)$
\end{remark}
\section{Matrix elements of the quantum cat map: Fluctuations in short
windows}\label{sec:P}
In this section we study the fluctuations of
the matrix elements of quantum cat map about their limit. We prove
the following theorem
\begin{thm}\label{thm:mtrx_fluct_3rd_mmnt}
Let $A\in SL_2(\bZ)$ be a hyperbolic matrix satisfying $A\equiv
I\pmod2$. Fix $f\in C^\infty(\bT^2)$ of zero mean. Assume
$L<2\ord(A,N)$, then as $N\to\infty$ through split primes satisfying
$\ord(A,N)/N^{2/3}\to\infty$, the third moment of $P(\theta)$
satisfies
\begin{equation}\label{eq:3rd moment result}
\int_0^1\left(\sqrt LP(\theta)\right)^3d\theta=O(\frac N{L^{3/2}})
\end{equation}
\end{thm}
We begin the proof by a reduction to the computation of mixed
moments of $P_\nu(\theta)$. We show that it suffice to prove the
following proposition
\begin{prop}\label{prop:mtrx_flct_mxd_mmnt}
Let $A\in SL_2(\bZ)$ be a hyperbolic matrix satisfying $A\equiv
I\pmod2$. Fix $0\ne\nu_1,\nu_2,\nu_3\in\bZ$. Then under the
conditions of theorem \ref{thm:mtrx_fluct_3rd_mmnt}
\begin{equation}\label{eq:3rd_mxd moment result}
\int_0^1\left(P_{\nu_1}(\theta)P_{\nu_2}(\theta)P_{\nu_3}(\theta)\right)^3d\theta=O(\frac
N{L^3})
\end{equation}
\end{prop}
Theorem \ref{thm:mtrx_fluct_3rd_mmnt} is a consequence of this
proposition as follows:\\
Let $f(x)=\sum_{n\in\bZ^2}\hat f(n)e(nx)\in C^{\infty}(\bT^2)$.
Write $f(x)=f_N(x)+f_R(x)$ where
$$f_N(x)=\sum_{\|n\|<N^{1/4}}\hat
f(n)e(nx),f_R(x)=\sum_{\|n\|\geq N^{1/4}}\hat f(n)e(nx)$$ Then we
have that $P(\theta)=P_N(\theta)+P_R(\theta)$ correspondingly. By
the fast decay of the Fourier coefficients of $f(x)$,
$\|P_R(\theta)\|_\infty=O(\frac1{N^\infty})$, and by corollary
\ref{cor:sharp expansion} we have that
$$
P_N(\theta)=\sum_{\nu\in\bZ}f_N^\sharp(\nu)P_\nu(\theta)=\sum_{|\nu|<N}f^\sharp(\nu)P_\nu(\theta)+O(N^{-\infty})\\
$$
again by the fast decay of the Fourier coefficients. By
Cauchy-Schwartz inequality
$$
\int_0^1P^3(\theta)d\theta=\int_0^1\left(P_N(\theta)+P_R(\theta)\right)^3d\theta=\int_0^1P_N^3(\theta)d\theta+O_f(\frac1{N^\infty})
$$
Now
\begin{equation*}
\int_0^1P_N^3(\theta)d\theta=\sum_{\nu_1,\nu_2,\nu_3\in\bZ}f^\sharp(\nu_1)f^\sharp(\nu_2)f^\sharp(\nu_3)\int_0^1P_{\nu_1}(\theta)P_{\nu_2}(\theta)P_{\nu_3}(\theta)d\theta
\end{equation*}
which proves theorem \ref{thm:mtrx_fluct_3rd_mmnt} by proposition
\ref{prop:mtrx_flct_mxd_mmnt}.
\subsection{Proof of proposition \ref{prop:mtrx_flct_mxd_mmnt}}
Denote
\begin{equation*}
H_3(t_1,t_2)=\sum_{l_1,l_2\in\bZ}\^h(\frac{t_1+l_1\ord}L)\^h(\frac{t_2+l_2\ord}L)\^h(\frac{-t_1-t_2-(l_1+l_2)\ord}L)
\end{equation*}
Expanding the Fourier expansion of $P_{\nu}(\theta)$ and calculating
the integral, we get that
\begin{equation}\label{eq:3rd_moment_cal}
\int_0^1P_{\nu_1}(\theta)P_{\nu_2}(\theta)P_{\nu_3}(\theta)d\theta=\frac1{L^3}\sum_{\tau_1,\tau_2\pmod{\ord(A,N)}}H_3(\tau_1,\tau_2)e(\bar2\frac{v(A^{\tau_1},A^{\tau_2})}N)
\end{equation}
where
$$v(g_1,g_2)=q(k_1;g_1)+q(k_2;g_2)+q(k_3;g_1g_2)$$
with $\omega(k_i,k_iA)=\nu_i,i=1,2,3$. Since $H_3$ is periodic with
period $\ord(A,N)$, we can write it as follows:
\begin{equation}\label{gamma_3_Fourier}
H_3(\tau_1,\tau_2)=\sum_{j_1,j_2\pmod{\ord(A,N)}}\gamma(j_1,j_2)e(\frac{j_1\tau_1+j_2\tau_2}{\ord(A,N)})
\end{equation}
where
\begin{eqnarray*}
\gamma(j_1,j_2)=\frac1{\ord(A,N)^2}\sum_{\tau_1,\tau_2}\Gamma_3(\tau_1,\tau_2)e(\frac{-j_1\tau_1-j_2\tau_2}{\ord(A,N)})=\\
\frac{L^3}{\ord(A,N)^2}\int_0^1h_L(x-\frac{j_1}{\ord(A,N)})h_L(x-\frac{j_2}{\ord(A,N)})h_L(x)dx
\end{eqnarray*}
which are in particular positive. Plugging (\ref{gamma_3_Fourier})
in (\ref{eq:3rd_moment_cal}), and switching order of summation, we
get that the RHS of (\ref{eq:3rd_moment_cal}) is
\begin{eqnarray*}
&&\frac1{L^{3/2}}\hspace{-0.9cm}\sum_{j_1,j_2\pmod\ord(A,N)}\hspace{-0.9cm}\gamma(j_1,j_2)\sum_{\tau_1,\tau_2\pmod\ord(A,N)}e(\frac{j_1\tau_1+j_2\tau_2}{\ord(A,N)}+\frac{v(\lambda^{\tau_1},\lambda^{\tau_2})}{N})=\\
&&\frac1{L^{3/2}}\sum_{j_1,j_2\pmod\ord(A,N)}\gamma(j_1,j_2)S_3(j_1,j_2)
\end{eqnarray*}
where
\begin{equation}\label{def:exp_sum1}
S_3(j_1,j_2)=\sum_{\tau_1,\tau_2\pmod{\ord(A,N)}}e(\frac{j_1\tau_1+j_2\tau_2}{\ord(A,N)}+\frac{v(A^{\tau_i},A^{\tau_2})}{N})
\end{equation}
Proposition \ref{prop:mtrx_flct_mxd_mmnt} will follow by showing
that there exists $M>0$ such that for any $j_1,j_2\pmod\ord(A,N)$
$|S_3(j_1,j_2)|\leq MN$. To show this we complete the sum to an
exponential sum over the group $\CA(N)$. For a pair of characters
$\chi_1,\chi_2$ of $\CA(N)$, and an additive character $\psi$ of
$\bF_N$ , set (as in \eqref{eq:exp_sum_3rd_mmnt_explct})
\begin{equation}\label{def:exp_sum2}
\mathcal
S(\chi_1,\chi_2;\psi)=\sum_{x,y\in\CA(N)}\;\!\!'\chi_1(x)\chi_2(y)\psi(v(x,y))
\end{equation}
% First notice
%that $e(\frac{j\tau}{\ord(A,N)})$ is a multiplicative character of
%the group generated by $\lambda$.
 Let $\psi$ be the additive character satisfying $\psi(1)=e(\frac{\nu_1}N)$ (recall $\nu_1\ne0$)), and let $g$ be a generator of $\CA(N)$ such that $g^r=A$, where $r=\frac{N}{\ord(A,N)}$. Denote by $\chi_0$ the character
of $\CA(N)$ satisfying $\chi_0(g)=e(\frac1{N-1})$. We therefore get
that $e(\frac{j\tau}{\ord(A,N)})=\chi_0^j(A^\tau)$. Writing the
indicator function of the subgroup generated by $A$ as
$$
\dsone_{A}(y) = \frac N{\ord(A,N)} \sum_{\substack{\theta\in
\CA(N)\\ \theta(A)=1}} \theta(y)
$$
and $a=\nu_1^{-1}\nu_2,b=\nu_1^{-1}\nu_3$, we can write $S_3$ as
$$S_3=\frac{N^2}{\ord(A,N)^2}\sum_{\substack{\theta_1,\theta_2\in \CA(N)\\ \theta_i(A)=1}}
\mathcal S(\chi_1\theta_1,\chi_2\theta_2;\psi)
$$

\begin{prop}\label{prop:bound_exp_sum}
Let $\chi_1,\chi_2$ be characters of $\CA(N)$, $\psi$ an additive
character of $\bF_N$, and $v(x,y)$ as above. Then if $A$ is
diagonalizable over the finite field $\bF_N$ then there exists $M>0$
such that $|S_3(j_1,j_2)|\leq MN$.
\end{prop}
\begin{proof}
If $A$ is diagonalizable over $\bF_N$, then $\CA(N)\simeq\bF_N^*$.
Under this isomorphism the sum $\mathcal S(\chi_1,\chi_2;\psi)$
becomes
$$
\mathcal
S(\chi_1,\chi_2;\psi)\sum_{x,y\in\bF_N^*}\chi_1(x)\chi_2(y)\psi(\frac{1+x}{1-x}+a\frac{1+y}{1-y}+b\frac{xy+1}{xy-1})
$$
By corollary \ref{cor:3rd_mmnt_rqurd_bnd} for the field $k=\bF_N$,
multiplicative characters
$\chi_1^{j_1}\theta_1,\chi_1^{j_2}\theta_2$, and additive character
$\psi$ the claim follows
\end{proof}
\begin{cor}
Let $\chi_1,\chi_2$ be characters of $\CA(N)$, $\psi$ an additive
character of $\bF_N$, and $v(x,y)$ as above. Then there exists $M>0$
such that $|S_3(j_1,j_2)|\leq MN$.
\end{cor}
\begin{proof}
Consider the group $\CA(N)$ as the $\bF_N$ rational points of the
algebraic group $\{B\in SL_2(\overline{\bF_N}):AB-BA=0\}$. Denote by
$w_1,\dots,w_l$ the roots of the exponential sum $\mathcal
S(\chi_1,\chi_2;\psi)$. Denote by $\sigma$ the Frobenius
automorphism of $\bF_{N^2}/\bF_N$, and for $x\in\CA(N^2)$ denote by
$\mathcal N(x)=x\sigma(x)$. Then by Deligne's result we have that
$$
\mathcal
S_2(\chi_1,\chi_2;\psi)=\sum_{x,y\in\CA(N^2)}\chi_1\circ\mathcal
N(x)\chi_2\circ\mathcal N(y)\psi(\tr(v(x,y))
$$
satisfies
$$
\mathcal S_2=w_1^2+\cdots+w_l^2
$$
Over $\bF_{N^2}$ the matrix $A$ is diagonalizable, and therefore we can apply
proposition \ref{prop:bound_exp_sum}, and get that $|w_i^2|\leq
N^2,i=1,\dots,l$, hence $|w_i|\leq N$, and the corollary follows.
\end{proof} We now have that
\begin{equation}
\int_0^1\left(P(\theta)\right)^3d\theta=O\left(\frac{N}{L^3}\sum_{j_1,j_2}|\gamma(j_1,j_2)|\right)
\end{equation}
and since $\gamma(j_1,j_2)$ are positive we can drop the absolute
value and remain with $\sum|\gamma(j_1,j_2)|=\Gamma(0)=O(1)$ since
$L<2\ord(A,N)$. This concludes the proof.
\section{Hecke matrix elements:
Independence}\label{sec:Hecke_mtrx_lmnts}
In \cite{KR_Hecke} Kurlberg and Rudnick showed that for any
hyperbolic matrix $A\in SL_2(\bZ)$, and for any $N$ there exist a
basis $\{\psi_i\}_{i=1}^N$ of $U_N(A)$ called (the) Hecke basis
satisfying that for any smooth function $f\in C^{\infty}(\bT^2)$
\begin{equation}\label{eq:KR_Hecke_AQUE}
|\langle\OPN(f)\psi_i,\psi_i\rangle-\int_{\bT^2}f|\ll
N^{-1/4-\epsilon}
\end{equation}
This result was later improved by Hadani and Gurevich for $N$ prime
to
\begin{equation}\label{eq:HG_Hecke_AQUE}
|\langle\OPN(f)\psi_i,\psi_i\rangle-\int_{\bT^2}f|\ll N^{-1/2}
\end{equation}
It was later conjectured by them (\cite{KR_Annals}) that when
normalizing these matrix elements by the correct size of $N$, and
add together the Fourier coefficients that correspond to natural
symmetries of the system, the fluctuations become equidistributed
and independent in the semiclassical limit. In this section we prove
some agreement with this conjecture. To state the precise theorem we
start with a some background that was not covered in section
\ref{sec:background}. As a general reference we use
\cite{KR_Hecke,KR00,KR_Annals}
\subsection{Hecke Theory for the Quantum cat map}
For a hyperbolic matrix $A\in SL_2(\bZ)$ satisfying $A\equiv
I\pmod2$, and $N$ prime Let
\begin{equation}\label{def:Hecke_group}
\CA(N)=SO(Q,\bZ/N\bZ)=\{B\in SL_2(\bZ/N\bZ): AB=BA\pmod N\}
\end{equation}
where $Q(n)=\omega(n,nA)$. The set
$$
\{U_N(B):B\in\CA(N)\}
$$
Is called the Hecke operators, and a basis of joint eigenfunction of
all Hecke operators is called a Hecke basis.
\begin{lem}\label{lem:Hecke_lmnts_invariant}
Let $\{\psi_j\}_{j=1}^N$ be a Hecke basis, and let $m,n\in\bZ^2$
such that $Q(n)=Q(m)$, then for all sufficiently large primes $N$ we
have
$$
(-1)^{n_1n_2}\langle T_N(n)\psi_j,\psi_j\rangle=(-1)^{m_1m_2}\langle
T_N(m)\psi_j,\psi_j\rangle,\quad j=1,\dots,N
$$
\end{lem}
In light of this lemma, for $\nu\in\bZ$, and $\psi$ a Hecke
eigenfunction define
$$
Y_\nu(\psi)=\sqrt N(-1)^{n_1n_2}\langle T_N(n)\psi,\psi\rangle
$$
where $n\in\bZ^2$ is such that $Q(n)=\nu$ if it exists (This is well
defined by lemma \ref{lem:Hecke_lmnts_invariant}). With this
definition conjecture \ref{conj:KR_equidistribution conjecture} is
the same as
\begin{conj}
As $N\to\infty$ through primes, ,for any $\nu\in Z$, the normalized
matrix element $Y_\nu(\psi)$ has a limiting distribution of
$\tr(U_\nu)$ as in conjecture \ref{conj:KR_equidistribution
conjecture}. Moreover, the sequence
$$
\dots,Y_{-3}(\psi),\;Y_{-2}(\psi),\;Y_{-1}(\psi),Y_1(\psi),\;Y_2(\psi),\;Y_3(\psi),\dots
$$
converge to a sequence of IID random variables
\end{conj}
In this section we prove the following theorem, which is in
agreement with conjecture \ref{conj:KR_equidistribution conjecture}.
\begin{thm}\label{thm:mtrx_lmnt_5th_mmnt}
Let $0\ne\nu_1,\dots,\nu_5\in\bZ$. Let $N$ be a prime number, let
$\{\psi_j\}_{j=1}^N$ be a Hecke basis of $U_N(A)$, then
\begin{equation}\label{eq:mtrx_lmnt_5th_mmnt}
\frac1N\sum_{j=1}^NY_{\nu_1}(\psi_j)\dots
Y_{\nu_5}(\psi_j)\ll\frac1{\sqrt N}
\end{equation}
\end{thm}
\begin{remark}
We note that theorem \ref{thm:mtrx_lmns_ffth_mmnt_gnrl} is a consequence of this theorem, by similar arguments to those showing that theorem \ref{thm:mtrx_fluct_3rd_mmnt} is a consequence of proposition \ref{prop:mtrx_flct_mxd_mmnt}.
\end{remark}
\subsection{Averaging operator} For $n\in\bZ^2$ let
$$
D(n)=\frac1{\CA(N)}\sum_{B\in\CA(N)}T_N(nB)
$$
The following lemma shows that this averaging operator is
essentially diagonal with respect to a Hecke basis (for proof see
lemma 7 in \cite{KR_Annals}).
\begin{lem}\label{lem:D_is_diagonal}
Let $\tilde D$ be the matrix obtained when expressing $D(n)$ in
terms of the Hecke eigenbasis. Then $\tilde D$ has the form
$$
\tilde D =\begin{pmatrix} D_{11}& D_{12}& 0& 0&\dots& 0\\ D_{21}&
D_{22}& 0& 0& \dots& 0\\ 0& 0& D_{33}& 0&\dots& 0\\ 0& 0& 0&
D_{44}&\dots& 0\\ \vdots& \vdots & \vdots & \vdots& \ddots& \vdots\\
0& 0& 0& 0&\dots & D_{NN}\end{pmatrix}
$$
such that
$$ |D_{ij}|\ll N^{-1/2}
$$
 for $1\leq i, j\leq 2$.
\end{lem}
\begin{lem}\label{lem:Hecke_mmnts_r_tr}
Let $\{\psi_j\}_{j=1}^N$ be a Hecke basis of $\mathcal H_N$, and let
$0\ne n_1,\dots,n_5\in\bZ^2$. Then
$$
\sum_{j=1}^N\langle T_N(n_1)\psi_j,\psi_j\rangle\cdots\langle
T_N(n_1)\psi_j,\psi_j\rangle=\tr(D(n_1)\cdots D(n_5))+O(N^{-5/2})
$$
\end{lem}
\begin{proof}
By lemma \ref{lem:D_is_diagonal}, we have that
$$
\tr(D(n_1)\cdots D(n_5))=\sum_{j=3}^ND(n_1)_{jj}\cdots
D(n_5)_{jj}+\tr(A_1\dots A_5)
$$
where $A_1,\dots,A_5$ are $2\times2$ matrices defined by
$$(A_k)_{ij}=(D(n_k))_{ij}\quad,1\leq i,j\leq2.$$ By lemma
\ref{lem:D_is_diagonal}
$$
\tr(A_1\dots A_5)=D(n_1)_{11}\cdots D(n_5)_{11}+D(n_1)_{22}\cdots
D(n_5)_{22}+O(N^{-5/2})
$$
and by definition of $D(n)$ the proof is concluded.
\end{proof}
\subsection{Proof of theorem
\ref{thm:mtrx_lmnt_5th_mmnt}}\label{subsec:prf_5th_mmnt_mtrx}
The first step of the proof is to reduce the required moment into an
exponential sum.
\begin{lem}\label{lem:5th_mmnt_mtrx_lmnt_exp_sum}
Choose $n_1,\dots,n_5$ such that $Q(n_i)=\nu_i$. Then For
$N>N_0(\nu_1,\dots,\nu_5)$ we have
\begin{equation}\label{eq:5th_mmnt_mtrx_lmnt_exp_sum}
\frac1N\sum_{j=1}^NV_{\nu_1}(\psi_j)\cdots
V_{\nu_5}(\psi_j)=\frac{N^{5/2}}{|\CA(N)|^5}{\hspace{-.8cm}}\sum_{\substack{B_1,\dots,B_5\in\CA(N)\\n_1B_1+\dots+n_5B_5=0\pmod
N}}{\hspace{-1.5cm}}e(\frac{\bar2u(B_1,\dots,B_5)}{N})+O(N^{-1})
%N}
\end{equation}
where $\bar2$ is the inverse of $2$ mod $N$, and
$$
u(B_1,B_2,B_3,B_4,B_5)=\sum_{1\leq i<j\leq5}\omega(n_iB_i,n_jB_j)
$$
\end{lem}
\begin{proof}
By definition we have that
\begin{eqnarray*}
&&\frac1N\sum_{j=1}^NV_{\nu_1}(\psi_j)\cdots
V_{\nu_5}(\psi_j)=\\
&&N^{3/2}\epsilon(n_1)\dots\epsilon(n_5)\sum_{j=1}^N\langle
T_N(n_1)\psi_j,\psi_j\rangle\cdots \langle
T_N(n_5)\psi_j,\psi_j\rangle=\\
&&N^{3/2}\epsilon(n_1)\dots\epsilon(n_5)\tr(D(n_1)\cdots
D(n_5))+O(N^{-1})
\end{eqnarray*}
where the last equality is by lemma \ref{lem:Hecke_mmnts_r_tr}. By
definition of $D(n)$, and by \eqref{eq:trace_of_TN},\eqref{eq:CCRN}
we have
\begin{eqnarray*}
&&\tr(D(n_1)\cdots
D(n_5))=\frac1{|\CA(N)|^5}\sum_{B_1,\dots,B_5\in\CA(N)}T_N(n_1B_1)\cdots
T_N(n_5B_5)=\\
&&\frac
N{|\CA(N)|^5}\epsilon(n_1B_1)\dots\epsilon(n_5B_5)\sum_{\substack{B_1,\dots,B_5\in\CA(N)\\n_1B_1+\dots+n_5B_5=0\pmod
N}}e(\frac{\bar2u(B_1,\dots,B_5)}{N})
\end{eqnarray*}
and since if $B\in\CA(N)$ then $B\equiv I\pmod2$, we have that
$\epsilon(n_i)=\epsilon(n_iB_i)$, which concludes the proof.
\end{proof}
Theorem \ref{thm:mtrx_lmnt_5th_mmnt} is now a consequence of the
following proposition
\begin{prop}\label{prop:5th_mmnt_mtrx_bnd}
If $Q(n_i)\ne0,i=1,\dots,5$, then for sufficiently large split prime
$N$
\begin{equation}\label{eq:5th_mmnt_dbl_sum}
\sum_{\substack{B_1,\dots,B_5\in\CA(N)\\n_1B_1+\dots+n_5B_5=0\pmod
N}}e(\frac{\bar2u(B_1,\dots,B_5)}{N})\ll N^2
\end{equation}
\end{prop}
\begin{proof}
By assumption there exist $M\in SL_2(\bF_N)$ such that $MAM^{-1}=D$
where $D$ is diagonal. In this case the group $\CA(N)$ is conjugated
to the group
$$
\CA(N)^M=\{MBM^{-1}:B\in\CA(N)\}=\{\begin{pmatrix}x&
\\&x^{-1}\end{pmatrix}:x\in\bF_N^*\}
$$
we can therefore write the sum as
\begin{equation}\label{eq:5th_mmnt_bnd_prf1}
\sum_{\substack{D_1,\dots,D_5\in\CA(N)^M\\n_1D_1+\dots+n_5D_5=0\pmod
N}}e(\frac{\bar2u(MB_1M^{-1},\dots,MB_5M^{-1})}{N})
\end{equation}
Writing $D_i=\begin{pmatrix}x_i&\\&x_i^{-1}\end{pmatrix}$ and
$m^{(i)}=(m_1^{(i)},m_2^{(i)}):=n_iM$, we have that
\begin{eqnarray*}
&&u(MB_1M^{-1},\dots,MB_5M^{-1})=\sum_{1\leq
i<j\leq5}\omega(n_iMD_i,n_jMD_j)=\\
&&\sum_{q\leq
i<j\leq5}\left(m_1^{(i)}m_2^{(j)}x_ix_j^{-1}-m_2^{(i)}m_1^{(j)}x_jx_i^{-1}\right)
\end{eqnarray*}
and the condition $n_1B_1+\dots+n_5B_5=0\pmod N$ becomes
\begin{eqnarray*}
\left\{\substack{m_1^{(1)}x_1+\dots+m_1^{(5)}x_5=0\\
m_2^{(1)}x_1^{-1}+\dots+m_2^{(5)}x_5^{-1}=0}\right\}=\left\{\substack{m_1^{(1)}m_2^{(1)}x_1+\dots+m_1^{(5)}m_2^{(5)}x_5=0\\
x_1^{-1}+\dots+x_5^{-1}=0}\right\}
\end{eqnarray*}
where the equations are in $\bF_N$, and the last equality is by the
change of variables $x_i\mapsto m_2^{(i)}x_i$. Denote
$A_i=m_1^{(i)}m_2^{(i)}$, (notice that these are nonzero since $n_i$
is not an eigenvector of $A$, as $Q(n_i)\ne0$)
\eqref{eq:5th_mmnt_bnd_prf1} is now
\begin{equation}\label{eq:5th_mmnt_bnd_prf2}
\sum_{\substack{0\ne
x_1,\dots,x_5\in\bF_N\\A_1x_1+\dots+A_5x_5=0\\
x_1^{-1}+\dots+x_5^{-1}=0}}e(\frac{\tilde h_A(x_3,x_4,x_5)}{N})
\end{equation}
where
$$
\tilde
h_A(x_3,x_4,x_5)=\frac1{x_3}+\frac1{x_4}+\frac1{x_5}+A_3x_3\left(\frac1{x_4}+\frac1{x_5}\right)+\frac{A_4x_4}{x_5}
$$
as in lemma \ref{lem:5th_mmnt_is_exp_sum}. By proposition
\ref{prop:5th_mmnt_exp_sum_bnd} the result now follows.
\end{proof}
\begin{cor}
Let $n_1,\dots,n_5\in\bZ^2$ such that $Q(n_i)\ne0,i=1,\dots,5$, then
for all primes $N$ not dividing $\tr(A)^2-4$ the inequality in
\eqref{eq:5th_mmnt_dbl_sum} holds
\end{cor}
\begin{proof}
By proposition \ref{prop:5th_mmnt_mtrx_bnd} we have that if $A$ is
diagonalizable over $k$, then
$$
S(A):=\sum_{\substack{B_1,\dots,B_5\in\CA(N)\\n_1B_1+\dots+n_5B_5=0\pmod
N}}e(\frac{\bar2u(B_1,\dots,B_5)}{N})\ll |k|^2
$$
and in particular all roots of the corresponding L-function
$L(S(A))$ are of absolute value at most $|k|^2$. We now show that
the diagonalizable condition may be dropped using "base change". For
a field $k$ consider the algebraic variety over the algebraic
closure $\overline k$ of $k$
 defined by
$$
\CA(\overline k)=\{B\in SL_2(\overline k):BA=AB\}
$$
and for $n_1,\dots,n_5\in k^2$ the subvariety of $\CA(\overline
k)^5$ defined by
$$
V=\{(B_1,\dots,B_5)\in\CA(\overline k)^5:n_1B_1+\dots+n_5B_5=0\}
$$
For any finite extension of $k$ of degree $\nu$, $k_\nu$, let
$$
V(k_\nu)=\{(B_1,\dots,B_5)\in\CA(k_\nu):(B_1,\dots,B_5)\in V\}
$$
be the $k_\nu$-rational points of $V$. Let $\omega_1,\dots,\omega_l$
be the roots of $L(S(A))$. Then by Deligne's result we have that
$$
\sum_{(B_1,\dots,B_5)\in
V(k_\nu)}e(\frac{\tr_{k_\nu/k}(\bar2u(B_1,\dots,B_5))}{N})=\omega_1^\nu+\dots+\omega_l^\nu
$$
If $\tr(A)^2\ne4$, then for $\nu=2$, $A$ is diagonalizable over
$k_\nu$, and hence by proposition \ref{prop:5th_mmnt_mtrx_bnd}
$|\omega_i^2|\leq |k|^4,i=1,\dots,l$, and therefore
$|\omega_i|\leq|k|^2$ which concludes the proof of proposition
\ref{prop:5th_mmnt_mtrx_bnd} for all primes $N$ that do not divide
$\tr(A)^2-4$.
\end{proof}
\section{Discussion}\label{sec:discussion}
\subsection{Matrix elements and exponential sums}
The connection between the matrix elements of the cat map and the
family of exponential sum $F(\chi;\psi)$, was observed previously by
Kurlberg and Rudnick, Gurevich and Hadani, and Kelmer
(\cite{KR_Annals,GH,Kelmer}). In \cite{Kelmer}, Kelmer shows that
the Hecke matrix element corresponding to the (non quadratic)
character $\chi$ of $\CA(N)$ is in fact of the form of
$F(\chi;\psi)$, hence at least in the split case they coincide.
Therefore conjecture \ref{conj:KR_equidistribution conjecture} can
be interpreted in the split case as prediction to the value
distribution of this family as $\chi$ varies. Conjecture
\ref{conj:exponential_sum} is a generalization of this conjecture
and, predicts that the action of the group of characters on this
family has a '{{\it{mixing type}}' behaviour (conjecture
\ref{conj:exponential_sum}.2). Agreement with these  predictions can
be seen in figure \ref{fig:high_moments}.
%\begin{figure}[h]\label{fig:distr}
%   \centerline{ \includegraphics[width=14cm,height=8cm]{distribution.eps}}
%   \caption{Sato Tate distribution}
%       \end{figure}
\begin{figure}[ht]
   \centerline{ \includegraphics[width=14cm,height=8cm]{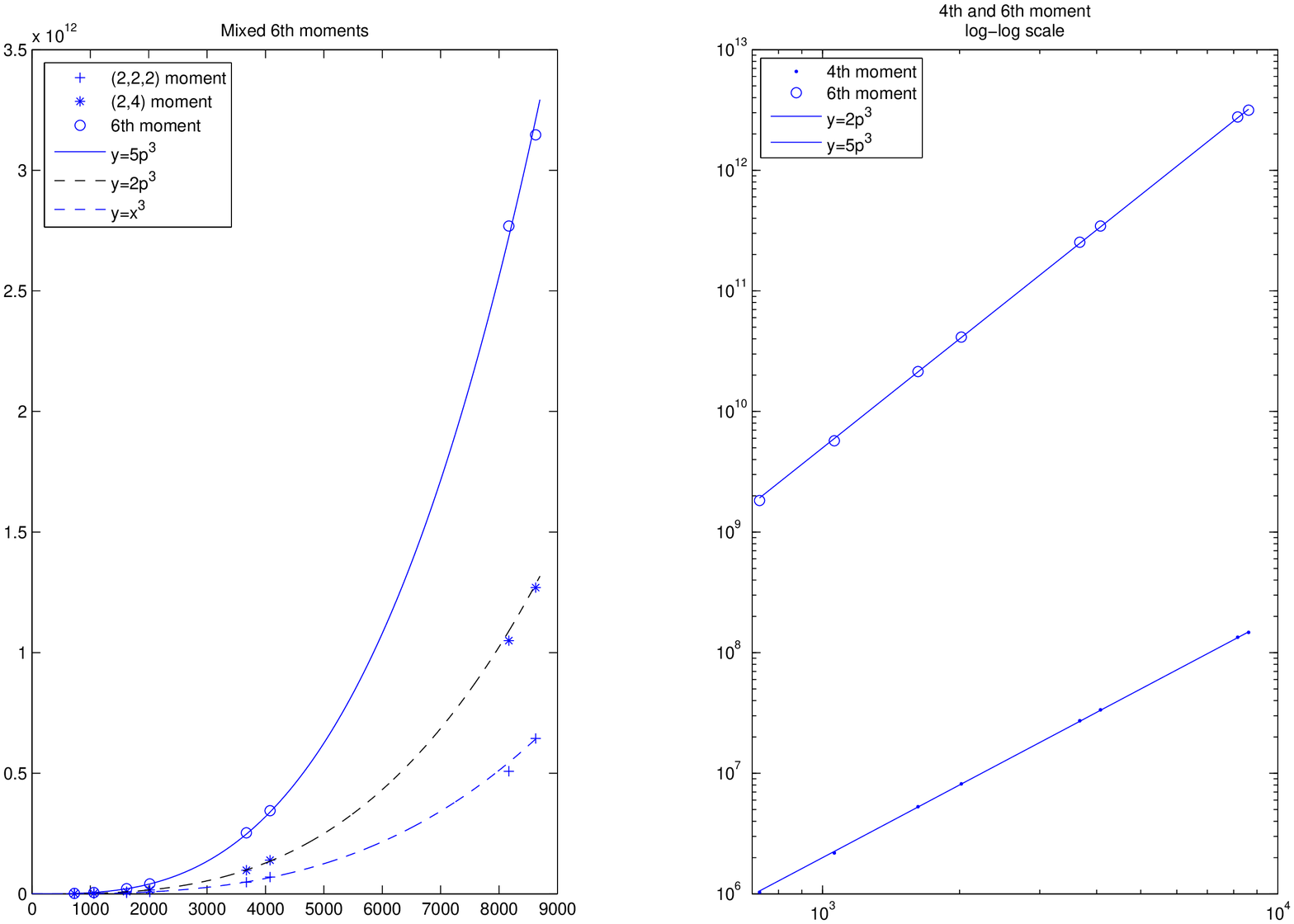}}
   \caption{High mixed moments}
   \label{fig:high_moments}
       \end{figure}
It shows high agreement of the numerical plots of mixed sixth
moments, and fourth moment. In the left part moments of type
$$
\frac1{p-1}\sum_{\chi}F(\chi_1\chi;\psi)\cdots F(\chi_6\chi;\psi)
$$
in different pairing. In case all characters $\chi_1,\dots,\chi_6$
are equal, it shows asymptotic growth of $5p^3$, and in case the
characters are split into subset of two and four equal characters (3
different pairs) we see asymptotic growth of $2p^3$ (respectively
$p^3$). This shows agreement with conjecture
\ref{conj:exponential_sum} recalling that if $X$ is a random
variable with Sato Tate distribution, then
$$
\mathbb E(X^{2n})=\begin{cases}1 &n=1\\2 &n=2\\5&n=3\end{cases}
$$

 It is a generic assumptions on matrix
elements that they behave independently with respect to the
eigenfunctions. When translating conjecture
\ref{conj:exponential_sum} to the language of Hecke matrix elements,
this independence behaviour appears as follows. In the case where
$N$ is prime, one can define a "product law" for the Hecke
eigenfunctions, by parameterizing the Hecke eigenfunction using the
characters of $\CA(N)$, $\psi_\chi$. We define for every character
$\chi_1$ of $\CA(N)$ the following operator
\begin{eqnarray*}
M_\chi:\mathcal H_N&\to&\mathcal H_N\\
\psi_\chi&\mapsto&\psi_{\chi_1\chi}
\end{eqnarray*}
Conjecture \ref{conj:exponential_sum} is therefore: for fixed
$l\in\bN$ and for any prime $N$ choose $l$ characters
$\chi_1,\dots,\chi_l$ of $\CA(N)$. Then as $N\to\infty$ through
primes
$$
\frac1{\CA(N)}\sum_{\chi\in\CA(N)}Y_{\nu_1}(L_{\chi_1}(\psi_\chi))\cdots
Y_{\nu_l}(L_{\chi_l}(\psi_\chi))\to\mathbb
E(\tr(U_{\nu_1})\cdots\tr(U_{\nu_l}))
$$
where $U_\nu$ are as in conjecture \ref{conj:KR_equidistribution
conjecture} independent for $\nu_i\ne\nu_j$.
\subsection{Fluctuations in short windows}
The expected independence behaviour of the matrix elements suggests
more on the fluctuations in short windows. Since at every point we
sum matrix elements related to different eigenvalues (and hence they
have independent behaviour), a Gaussian limiting distribution may
appear. The following figures show agreement with this heuristic.
\begin{figure}[ht]
   \centerline{ \includegraphics[width=14cm,height=10cm]{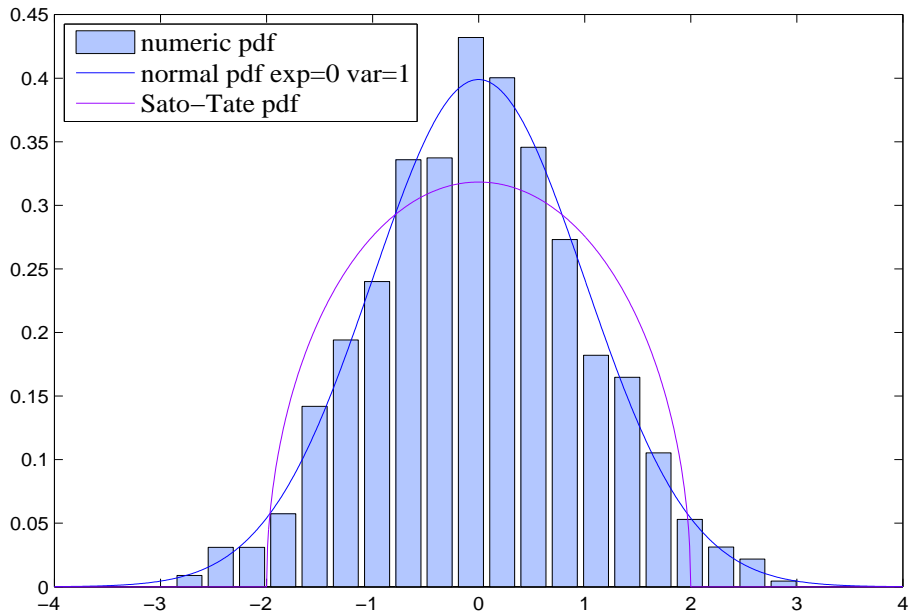}}
   \caption{$P(\theta)$ distribution, $f(x)=e(x+y)$}
   \label{fig:P_theta_var1}
       \end{figure}
\begin{figure}[ht]
   \centerline{ \includegraphics[width=14cm,height=8cm]{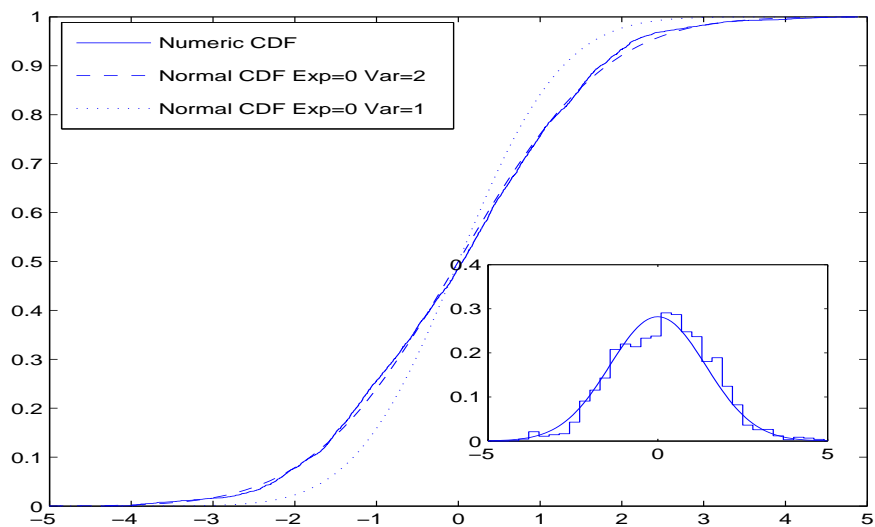}}
   \caption{$P(\theta)$ distribution, $f(x)=e(x+y)+e(2x+y)$}
   \label{fig:P_theta_var2}
       \end{figure}
In figure \ref{fig:P_theta_var1} comparison between the distribution
of $P(\theta)$, normal distribution, and Sato-Tate distribution is
displayed. It shows that the distribution agrees with normal
distribution rather than Sato-Tate. In figure \ref{fig:P_theta_var2}
the function $f(x)$ is a trigonometrical polynomial with exponents
that give two different values for $Q(n)$. We thus see that the
variance is 2 rather than 1 as in figure \ref{fig:P_theta_var1}.
This is in fact the result shown in \cite{KRR}. In fact, assuming
conjecture \ref{conj:exponential_sum}, it is possible to show the
following theorem (see \cite{Ros}):
\begin{thm}\label{conj:gaussian_P}
Let $A\in SL_2(\bZ)$ be a unimodular matrix with distinct
eigenvalues such that $A\equiv I\pmod2$. Fix a smooth function $f\in
C^\infty(\bT^2)$. Assume that for every $\epsilon>0$
$N^{1-\epsilon}\ll L$, and $\frac{L}{\ord(A,N)}\to0$ as $N$ goes to
infinity through primes, then, assuming conjecture
\ref{conj:exponential_sum}, $P(\theta)$ has Gaussian limiting
distribution with mean 0 and $var(P)=\sum_{\nu\in\bZ}f^\sharp(\nu)$
\end{thm}
\begin{remark}
Notice that the conditions $L/\ord(A,N)\to0$ and $N^{1-\eps}\ll L$
imply an assumption on the size of $\ord(A,N)$, however, assuming
GRH this assumption is valid for most primes (c.f.\cite{Kur03}).
\end{remark}
This result is in some agreement with predicted results on generic
systems. It says that once the arithmetic symmetries of the systems
are grouped together, the resulting desymetrized components have a
generic Gaussian limiting behaviour. However we should notice that
when the size of the window becomes too short, the function
$P(\theta)$ no longer consists of sums of matrix elements
corresponding to different eigenvalues, and in cases where the order
of $A$ modulo $N$ is maximal it studies the matrix elements
distribution and we no longer expect normal distribution but rather
Sato-Tate as is shown in figure \ref{fig:P_theta_ST}
\begin{figure}[ht]
   \centerline{ \includegraphics[width=14cm,height=9cm]{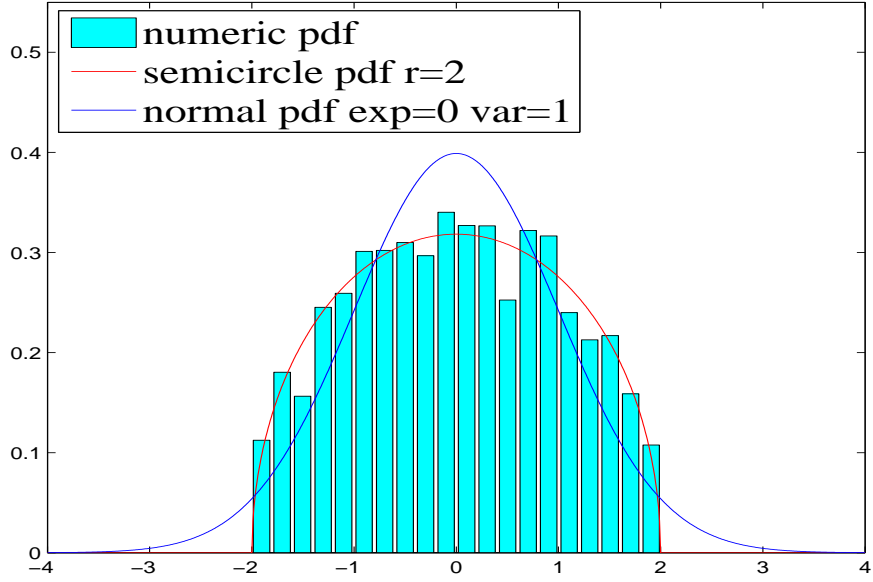}}
   \caption{$P(\theta)$ distribution, Very short window}
   \label{fig:P_theta_ST}
       \end{figure}
       \newpage
\appendix
\section{Proofs of irreducibility}\label{ap:irred_proofs}
Let $k$ be a finite field, and $0\ne A_1,\dots,A_5\in k$. Denote by
\begin{equation*}
{\bf{V}}(A)=\left\{0\ne a_1,a_2,a_3,a_4\in
k:\substack{A_1a_1+A_2a_2+A_3a_3+A_4a_4+A_5=0\\a_1^{-1}+a_2^{-1}+a_3^{-1}+a_4^{-1}+1=0}\right\}
\end{equation*}
and by
\begin{eqnarray*}
{\bf{V}}(A,\nu)=\left\{0\ne
a_1,a_2,a_3,a_4\in k_\nu:\substack{A_1a_1+A_2a_2+A_3a_3+A_4a_4+A_5=0\\a_1^{-1}+a_2^{-1}+a_3^{-1}+a_4^{-1}+1=0}\right\}\\
\overline{{\bf{V}}}(A)=\left\{0\ne a_1,a_2,a_3,a_4\in\overline
k:\substack{A_1a_1+A_2a_2+A_3a_3+A_4a_4+A_5=0\\a_1^{-1}+a_2^{-1}+a_3^{-1}+a_4^{-1}+1=0}\right\}
\end{eqnarray*}
the rational points in any finite extension $k_\nu$ of $k$, and the
points in the algebraic closure $\overline k$. Let
\begin{equation*}
\tilde h_A(a_2,a_3,a_4)=\sum_{2\leq
i<j\leq4}\left(\frac{A_ia_i}{a_j}-\frac{A_ja_j}{a_i}\right)+\sum_{i=1}^4A_ia_i-A_5\sum_{i=2}^4a_i^{-1}
\end{equation*}
We prove the following lemmas
\begin{lem}\label{lem:5th_mmnt_srfce_irred}
For $0\ne A_1,\dots,A_5\in k$, the variety $\overline{{\bf{V}}}(A)$
has one irreducible component of dimension 2
\end{lem}
and
\begin{lem}\label{lem:5th_mmnt_fibers_irred}
Except for 14 values of $C\in\overline k$ the fibers
$$\tilde
h_A^{-1}(C)\subset\overline{\bf{V}}(A)
$$
are irreducible
\end{lem}
\subsection{Proof of lemma \ref{lem:5th_mmnt_srfce_irred}}
We prove the lemma by counting the number of points on
${\bf{V}}(A,\nu)$ for every $\nu\in\bN$. By Lang-Weil theorem, any
irreducible component surface has $|k_\nu|^{2}+O(|k_\nu|^{3/2})$
points on it, and therefore by showing this we prove the lemma. For
a nontrivial additive character $\psi$ of $k_\nu$, and $0\ne a,b\in
k_\nu$ denote by $Kl(a,b)$ the Kloosterman sum
$$
Kl(a,b)=\sum_{0\ne x\in k_\nu}\psi(ax+bx^{-1})
$$
\begin{prop}
Let $\psi$ be a nontrivial additive character of $k_\nu$. Then
$$
\sharp {\bf{V}}(A)-|k_\nu|^2=\frac1{|k_\nu|^2}\sum_{b\in
k_\nu}\left(\prod_{i=1}^5Kl(a_i,b)\right)+O(|k_\nu|^{-1})
$$
where $Kl(a,b)=\sum_{0\ne x}\psi(ax+bx^{-1})$, is the Kloosterman
sum.
\end{prop}
\begin{proof}
By the orthogonality relations of additive characters, we have that
\begin{eqnarray*}
&&\sharp{\bf{V}}(A)=\frac1{|k_\nu|^2}\sum_{a,b\in k_\nu}\sum_{\underline x\in\left(k_\nu^*\right)^4}\psi(af_1(\underline x)+bf_2(\underline x))=\\
&&|k_\nu|^2+\frac1{|k_\nu|^2}\sum_{b\in k_\nu^*}\psi(-b)\prod_{i=1}^4Kl(0,b)+\\
&&\frac1{|k_\nu|^2}\sum_{a,b\in
k_\nu,a\ne0}\psi(aa_5+b)\prod_{i=2}^5\left(Kl(aa_i,b)\right)+O(|k_\nu|^{-1})
\end{eqnarray*}
Using that $Kl(0,b)=-1,Kl(ac,b)=Kl(c,ab)$ we have that
\begin{equation*}
\sharp {\bf{V}}(A)-|k_\nu|^2=\frac1{|k_\nu|^2}\sum_{a\in
k_\nu^*}\sum_{b\in
k_\nu}\psi(-aa_5-b)\prod_{i=1}^4\left(Kl(a_i,ba)\right)^4+O(|k_\nu|^{-1})
\end{equation*}
and under the change of variable $b\mapsto b/a$, we get
\begin{equation*}
\sharp {\bf{V}}(A)-|k_\nu|^2=\frac1{|k_\nu|^2}\sum_{b\in
k_\nu}\prod_{i=1}^5\left(Kl(a_i,b)\right)+O(|k_\nu|^{-1})
\end{equation*}
and now using Weil's bound $Kl(a_i,b)\leq2N^{\nu/2}$, we get the
bound
$$
\sharp{{\bf{V}}}(A)=|k_\nu|^2+O(|k_\nu|^{3/2})
$$
which proves the lemma.
\end{proof}
\subsection{Proof of lemma \ref{lem:5th_mmnt_fibers_irred}}
We prove the irreducibility of the fibers by the following strategy:
For each curve $\tilde h_A^{-1}(C)$ we find a curve in the affine
plane $\mathbb{A}$ over $\overline k$ given by the zeros set of a
polynomial, such that (an open Zariski subset of) the fiber is
parameterized by (an open Zariski subset of) this plane curve.  We then show that the polynomial defining the plane curve is irreducible over $\overline k$ and thus proving the lemma.

For simplicity of notations we use the following notation: For a
polynomial $P(X_1,\dots,X_n)\in k[X_1,\dots,X_n]$ over a field $k$
we denote the zeros set of this polynomial by
$$
Z(p)=\{(a_1,\dots,a_n\in\mathbb A^n:P(a_1,\dots,a_n)=0\}
$$
and its complement by
$$
Y_p=\{(a_1,\dots,a_n)\in\mathbb A^n:P(a_1,\dots,a_n)\ne0\}
$$
For fixed $A_1,\dots,A_5,C\in k$ we define the following
polynomial $p(a_3,a_4)$:
\begin{eqnarray}\label{eq:polynomial definition}
&&p(a_3,a_4)=2a_4^2((B+C-2A_2)a_3+2A_4a_4)(A_3a_3+A_4a_4)+\\
\nonumber&&a_4(2(B-C-2A_2+2A_5)(A_4a_4^2+A_3a_3^2)+(B^2-C^2+D)a_3 a_4)+\\
\nonumber&&2A_5((B-C+2A_2)a_4+2A_3a_3)(a_3+a_4)
\end{eqnarray}
where
$B=A_2+A_3+A_4+A_5-A_1,D=-4A_2(A_3+A_4+A_5)+4(A_3A_4+A_3A_5+A_4A_5)$.
\begin{prop}\label{prop:birational map for fibers}
Let $p(a_3,a_4)$ be as above. Define
 \begin{eqnarray*}
g_1(x_3,x_4)&=&A_3x_3+A_4x_4+A_5\\
g_2(x_3,x_4)&=&(B+C)x_3 x_4+2A_4x_4^2+2A_5(x_4+x_3)
\end{eqnarray*}
and
 \begin{eqnarray*}
\tilde g_1(a_1,a_2,a_3,a_4)&=&A_3a_3+A_4a_4+A_5\\
\tilde g_2(a_1,a_2,a_3,a_4)&=&(B+C)a_3 a_4+2A_4a_4^2+2A_5(a_4+a_3)
\end{eqnarray*}
 Then
the following map
\begin{eqnarray*}
&&a_1=-\frac{A_2a_2(x_3,x_4)+A_3x_3+A_4x_4+A_5}{A_1}\\
&&a_2=\frac{-2x_3 x_4 g_1(x_3,x_4)}{g_2(x_3,x_4)}\\
&&a_3=x_3\\
&&a_4=x_4
\end{eqnarray*}
defines a bijection between ${\tilde h_A^{-1}(C)}\cap Y_{\tilde
g_1}\cap Y_{\tilde g_2}$ and $Z(p)\cap Y_{g_1}\cap Y_{g_2}$
%Let $f_1(x,a_2,a_3,a_4),f_2(x,a_2,z,a_4),f(1,-x,-a_2,-z,-a_4)$ be as above.
%Consider the set
%$$V(\nu,\vec f,C):=\left\{a_1,a_2,a_3,a_4\in\bF_{N^\nu}^*|\substack{f_1(a_1,a_2,a_3,a_4)=0\\f_2(a_1,a_2,a_3,a_4)=0\\f(1,-a_1,-a_2,-a_3,-a_4)-C=0}\right\}$$
%Then if $a_2z-a_2-z\ne0,a_2z+a_2+z\ne0,(1+C)a_2+2z\ne0,(2 a_2 + z - C z)\ne0$,
%then
%\begin{eqnarraa_2*}
%&&x=\frac{2a_2(a_2^2+a_2z+z^2)}{((1+C)a_2+2z)(a_2 z-a_2-z)}\\
%&&a_4=\frac{2z(a_2^2+a_2z+z^2)}{((1-C)z+2a_2)(a_2 z-a_2-z)}\\
%&&p(a_2,z)=0
%\end{eqnarraa_2*}
\end{prop}
\begin{proof}
It is straightforward to check that if $(x_3,x_4)\in Z(p)$, then
their lies in $\tilde h_A^{-1}(C)$. To show the other direction, we
consider the system of equations
\begin{eqnarray*}
&&A_1a_1+A_2a_2+A_3a_3+A_4a_4+A_5=0\\
&&\frac1{a_1}+\frac1{a_2}+\frac1{a_3}+\frac1{a_4}+1=0\\
&&\sum_{2\leq
i<j\leq4}\left(\frac{A_ia_i}{a_j}-\frac{A_ja_j}{a_i}\right)+\sum_{i=1}^4A_ia_i-A_5\sum_{i=2}^4a_i^{-1}=C
\end{eqnarray*}
Multiply the second equation by $A_1a_1a_2a_3a_4$ and substitute
$A_1a_1$ by\\ $-(A_2a_2-A_3a_3-A_4a_4-A_5)$ to get
\begin{eqnarray}\label{eq:param_prf1}
&&A_2a_2^2(a_3+a_4+a_3a_4)+A_3a_3^2(a_2+a_4+a_2a_4)+\\
&&\nonumber
A_4a_4^2(a_2+a_3+a_2a_3)+Ba_2a_3a_4+A_5(a_2a_3+a_2a_4+a_3a_4)=0.
\end{eqnarray}
Next we multiply the third equation by $a_2a_3a_4$ to get
\begin{eqnarray}\label{eq:param_prf2}
&&A_2a_2^2(a_3+a_4+a_3a_4)+A_3a_3^2(-a_2+a_4+a_2a_4)+\\
&&\nonumber A_4a_4^2(-a_2-a_3+a_2a_3)-A_5(a_2a_3+a_2a_4+a_3a_4)=0.
\end{eqnarray}
Subtract \eqref{eq:param_prf1} by \eqref{eq:param_prf2} to get:
\begin{equation*}\label{eq:eliminating y}
g_2(a_3,a_4)a_2+2a_4 a_3 g_1(a_3,a_4)=0
\end{equation*}
By assumption that $g_2(a_3,a_4)\ne0$ we get that
$$
a_2=\frac{-2 a_3 a_4g_1(a_3,a_4)}{g_2(a_3,a_4)}
$$
Use this expression for $a_2$ inside $\tilde h_A(a_2,a_3,a_4)-C=0$
to get
$$
\frac{p(a_3,a_4)}{2 a_4 g_2(a_3,a_4)}=0
$$
and therefore $a_3,a_4$ must satisfy $p(a_3,a_4)=0$.
\end{proof}
\begin{prop}\label{prop:poly irreducible}
Let $p(X_3,X_4)$ as in \eqref{eq:polynomial definition}. Then for
all $C\in\overline k$ but at most 14 values, the polynomial
$p(X_3,X_4)$ is irreducible over $\overline k$.
\end{prop}
\begin{proof}
We first denote the following homogeneous parts of $p(X_3,X_4)$
\begin{eqnarray*}
&p_4(X_3,X_4)=&2X_4^2((B+C-2a_2)X_3+2a_4X_4)(a_3X_3+a_4X_4)\\
&p_3(X_3,X_4)=&X_4(2(B-C-2(a_2-a_5))(a_4X_4^2+a_3X_3^2)+\\&&(B^2-C^2+D)X_3 X_4)\\
&p_2(X_3,X_4)=&2a_5((B-C+2a_2)X_4+2a_3X_3)(X_3+X_4)
\end{eqnarray*}
Let $q(X_3,X_4),r(X_3,X_4)$ be two polynomials satisfying\\
$p(X_3,X_4)=q(X_3,X_4)r(X_3,X_4)$, and denote their decomposition
into homogeneous parts, $q=q_0+q_1+q_2+q_3,r=r_0+r_1+r_2+r_3$.
Without loss of generality we assume $\deg{r}\leq\deg(q)$. Since the
homogeneous parts of $p$ are of degree 2,3,4 only this imposes a few
restrictions on $q,r$. We split the cases into 2 parts
\begin{enumerate}
\item
 {\bf{Case 1 $q_0\ne0$}}: If $q_0\ne0$ we get that $r_0=r_1=0$ since otherwise
the minimal degree of $qr<2$, moreover $r_2\ne0$. This implies that
$q_3=0,q_1\ne0$ (otherwise $\deg{qr}>4$, and there would not be a
homogeneous part of degree 3). It is left to check whether $q_2$
vanishes or not.
\begin{enumerate}
\item
{\bf{$q_2\ne0$}}: If $q_2\ne0$ then $r_3=0$ and we get that
$r=r_2=p_2$ is homogeneous of degree 2, and that $p_2$ divides $p$,
in particular $(X_3+X_4)$ divides $p(X_3,X_4)$. Considering this
composition in $\overline k(X_3)[X_4]$ this implies that $-X_4$ is a
root of $p(X_3,X_4)$ that is $p(X_3,-X_3)=0$. The coefficient of the
fourth degree of $p(X_3,-X_3)$ is then
$$
(B^2 - 2 a_3 (-2 a_2 + 2 a_5 + B - C) - 2 a_4 (-2 a_2 + 2 a_5 + B
-C) - C^2 + D)
$$
that vanishes for at most 2 values of $C$.
%and in particular $p_2$ divides both $p_3,p_4$. If $a_3\ne a_4$,
%this implies that either $(X_3+X_4)|((B+C-2a_2)X_3+2a_4X_4)$ or
%$((B-C+2a_2)X_4+2a_3X_3)|((B+C-2a_2)X_3+2a_4X_4)$ or
%$(a_3X_3+a_4X_4)$, and this occurs only if $B+C-2a_2=2a_4$,
%$B-C+2a_2=a_4$ or $(B-C+2a_2)(B+C-2a_2)-4a_3a_4=0$, which is a
%finite set of values of $C$. If $a_3=a_4$ then $(X_3+X_4)|p,p_4,p_2$
%and therefore $(X_3+X_4)|p_3$, however this again can only happen
%for at most 2 values of $C$.
\item
{\bf{$q_2=0$}}: If $q_2=0$, then $r_3\ne0$ and we have that
$q=q_0+q_1,r=r_2+r_3$, such that $q_0r_2=p_2,q_1r_3=p_4$ and
$q_1r_2+q_0r_3=p_3$. Without loss of generality we may assume
\begin{eqnarray*}
&&q_1=X_4,\;((B+C-2a_2)X_3+2a_4X_4)
{\mbox{,\;or}}\quad(a_3X_3+a_4X_4)\\
&&r_3=\frac{p_4}{q_1},\quad r_2=\frac{p_2}{q_0}
\end{eqnarray*}
If $q_1\ne X_4$ we get that $X_4$ divides $r_3,p_3$ and therefore
$X_4|r_2$ which is a contradiction. Therefore we get that $q_1=X_4$,
and
\begin{eqnarray*}
&&p_2(X_3,X_4)/q_0+q_0((B+C-2a_2)X_3+2a_4X_4)(a_3X_3+a_4X_4)=\\
&&2(B-C-2a_2+2a_5)(a_4X_4^2+a_3X_3^2)+(B^2-C^2+D)X_3 X_4
\end{eqnarray*}
comparing coefficients gives 3 equations for $C,q_0$ that have at
most 6 solutions for $C$.
\end{enumerate}
\item
{\bf{Case 2}} $q_0=0$: If $q_0=0$ then $q_1\ne0$ (otherwise
$\deg(r)\>\deg(q)$ that contradicts our assumption). This implies
that $r_0=0,r_1,q_2\ne0$. If $q_3\ne0$ then $r_2=0$ and hence we get
that $r_1$ divides $p,p_2,p_3,p_4$ which we saw above that can
happen for at most 2 values of $C$. We therefore get that
$q=q_1+q_2,r=r_1+r_2$ satisfying
$q_1r_1=p_2,q_1r_2+q_2r_1=p_3,q_2r_2=p_4$. Since $X_4$ does not
divide $p_2$ and does divide $p_3,p_4$ we find that $X_4$ must
divide $q_2,r_2$. We therefore assume without loss of generality,
that $q_2=X_4(a_3X_3+a_4X_4), r_2=p_4/q_2$, and
$(q_1,r_1)=(\mu(X_3+X_4),\frac1\mu(2a_5(B-C+2a_2)X_4+2a_3X_3))$ or
$(\mu(2a_5(B-C+2a_2)X_4+2a_3X_3),\frac1\mu(X_3+X_4))$. Comparing
coefficients again for $q_1r_2+q_2r_1=p_3$ we get 3 equations for
$C,\mu$ that have at most 6 solutions in $C$.
\end{enumerate}
Combining all restrictions for $C$ we get that if $C$ is outside a
set of cardinality at most 14 $p(X_3,X_4)$ is irreducible.
\end{proof}

\end{document}